\documentclass[11pt]{article}

\usepackage[hmargin={24mm,24mm},vmargin={30mm,35mm}]{geometry}
\usepackage{amsfonts}
\usepackage{amssymb}
\usepackage{amsthm}
\usepackage{bm}
\usepackage{mathtools}
\usepackage{latexsym}
\usepackage{graphicx}
\usepackage{booktabs}
\usepackage{pgfplots,pgfplotstable}
\usepackage{xcolor}
\usepackage{setspace}
\usepackage{url}
\usepackage{float}
\usepackage{subcaption}
\usepackage{enumitem}
\usepackage{cancel}

\usepackage[colorlinks,allcolors={blue},linkcolor={black}]{hyperref}

\usepackage[backend=bibtex,style=numeric-comp,maxnames=99,firstinits=true]{biblatex}

\numberwithin{equation}{section}

\usepackage[normalem]{ulem}
\newcounter{corr}
\definecolor{violet}{rgb}{0.580,0.,0.827}
\newcommand{\corr}[4][]{\typeout{Warning : a correction remains in page \thepage}
 \stepcounter{corr}
	      {\color{blue}\ifmmode\text{\,\sout{\ensuremath{#2}}\,}\else\sout{#2}\fi}
        {\color{green!50!black}#3}
        {\color{violet}#4}
}

\usepackage{authblk}
\newcommand{\email}[1]{\href{mailto:#1}{#1}}

\newtheorem{theorem}{Theorem}

\newtheorem{proposition}[theorem]{Proposition}
\newtheorem{lemma}[theorem]{Lemma}
\newtheorem{corollary}[theorem]{Corollary}

\newtheorem{assumption}{Assumption}

\newtheorem{remark}{Remark}

\newcommand{\bmrm}[1]{{\bm{{\rm #1}}}}

\newcommand{\bmb}{{\bm{b}}}

\newcommand{\bme}{{\bm{e}}}
\newcommand{\bmf}{{\bm{f}}}
\newcommand{\bmg}{{\bm{g}}}

\newcommand{\bmn}{{\bm{n}}}

\newcommand{\bmu}{{\bm{u}}}
\newcommand{\bmv}{{\bm{v}}}
\newcommand{\bmw}{{\bm{w}}}

\newcommand{\bmz}{{\bm{z}}}

\newcommand{\bmB}{\bm{B}}
\newcommand{\bmC}{\bm{C}}

\newcommand{\bmG}{\bm{G}}

\newcommand{\bmI}{\bm{I}}

\newcommand{\bmU}{\bm{U}}

\newcommand{\bmX}{\bm{X}}

\newcommand{\bmdelta}{{\bm{\delta}}}

\newcommand{\bmkappa}{{\bm{\kappa}}}

\newcommand{\bmpi}{{\bm{\pi}}}

\newcommand{\bmtau}{{\bm{\tau}}}

\newcommand{\calE}{\mathcal{E}}
\newcommand{\calF}{\mathcal{F}}

\newcommand{\calH}{\mathcal{H}}

\newcommand{\calM}{\mathcal{M}}
\newcommand{\calN}{\mathcal{N}}

\newcommand{\calT}{\mathcal{T}}

\newcommand{\bbP}{\mathbb{P}}

\newcommand{\bbR}{\mathbb{R}}

\newcommand{\rma}{{\rm{a}}}
\newcommand{\rmb}{{\rm{b}}}

\newcommand{\rmd}{{\rm{d}}}

\newcommand{\rmi}{{\rm{i}}}

\newcommand{\rms}{{\rm{s}}}
\newcommand{\rmt}{{\rm{t}}}

\newcommand{\rmA}{{\rm{A}}}

\newcommand{\rmD}{{\rm{D}}}

\newcommand{\rmF}{{\rm{F}}}
\newcommand{\rmG}{{\rm{G}}}

\newcommand{\rmT}{{\rm{T}}}

\newcommand{\eq}{ ={}& }
\newcommand{\lea}{ \le{}& }

\newcommand{\les}{ \lesssim{}& }

\newcommand{\minus}{ &{}- }

\newcommand{\nn}{\nonumber}
\newcommand{\nl}{\nn\\}

\newcommand{\defeq}{\vcentcolon=}

\newcommand{\deq}{ \defeq & \, }

\newcommand{\ul}[1]{\underline{#1}}
\newcommand{\ol}[1]{\overline{#1}}
\newcommand{\ulbm}[1]{\ul{\bm{#1}}}
\newcommand{\olbm}[1]{\ol{\bm{#1}}}

\newcommand{\CARD}[1]{\mbox{\textrm{Card}}\big(#1\big)}

\newcommand{\bdry}{\partial}

\newcommand{\Mh}[1][h]{\calM_{#1}}
\newcommand{\Th}[1][h]{\calT_{#1}}
\newcommand{\Fh}[1][h]{\calF_{#1}}
\newcommand{\Fhb}{\Fh^{\rmb}}
\newcommand{\Fhi}{\Fh^{\rmi}}

\newcommand{\T}{{T}}
\newcommand{\F}{{F}}
\newcommand{\FT}{{\Fh[\T]}}

\newcommand{\hT}{h_\T}
\newcommand{\hF}{h_\F} 
\newcommand{\hX}{h_X}

\newcommand{\bdryT}{{\bdry \T}}

\newcommand{\nor}{\bmn}
\newcommand{\norT}{\nor_{\T}}
\newcommand{\norF}{\nor_{\F}}
\newcommand{\norTF}{\nor_{\T\F}}

\def\R{\bbR}

\newcommand{\POLY}[1]{\bbP^{#1}}
\newcommand{\HS}[1]{H^{#1}}
\newcommand{\HONE}{\HS{1}}
\newcommand{\HONEzr}{\HONE_0}
\newcommand{\WSP}[1]{W^{#1}}
\newcommand{\LP}[1]{L^{#1}}
\newcommand{\LTWO}{\LP{2}}

\DeclareMathOperator*{\DIV}{div}
\DeclareMathOperator*{\CURL}{curl}

\newcommand{\norm}[2][]{\|#2\|_{#1}}
\newcommand{\seminorm}[2][]{|#2|_{#1}}
\newcommand{\brac}[2][]{(#2)_{#1}}

\newcommand{\SqBrac}[2][]{\Big[#2\Big]_{#1}}

\newcommand{\Brac}[2][]{\Big(#2\Big)_{#1}}

\newcommand{\SEMINORM}[2][]{\left|#2\right|_{#1}}
\newcommand{\BRAC}[2][]{\left(#2\right)_{#1}}

\newcommand{\piT}[1]{\pi_T^{#1}}
\newcommand{\piTzr}[1]{\piT{0, #1}}

\newcommand{\pih}[1]{\pi_h^{#1}}
\newcommand{\pihzr}[1]{\pih{0, #1}}

\newcommand{\bmpiT}[1]{\bmpi_T^{#1}}
\newcommand{\bmpiTzr}[1]{\bmpiT{0, #1}}
\newcommand{\bmpiTe}[1]{\bmpiT{1, #1}}

\newcommand{\bmpih}[1]{\bmpi_h^{#1}}
\newcommand{\bmpihzr}[1]{\bmpih{0, #1}}

\newcommand{\bmpiFT}[1]{\bmpi_\FT^{#1}}
\newcommand{\bmpiFTzr}[1]{\bmpiFT{0, #1}}

\newcommand{\rT}[1]{\bmrm{r}_T^{#1}}
\newcommand{\DT}[1]{\rmD_T^{#1}}
\newcommand{\GT}[1]{\bmG_T^{#1}}
\newcommand{\Gh}[1]{\bmG_h^{#1}}
\newcommand{\rh}[1]{\bmrm{r}_h^{#1}}

\newcommand{\bmUTk}{\ulbm{U}_{T}^{k}}
\newcommand{\bmUhk}{\ulbm{U}_{h}^{k}}

\newcommand{\bmUhkzr}{\ulbm{U}_{h,0}^{k}}
\newcommand{\bmUhkn}{\ulbm{U}_{h,\bmn}^{k}}

\newcommand{\bmIT}[1]{\ul{\bmI}_{T}^{#1}}
\newcommand{\bmITk}{\bmIT{k}}

\newcommand{\bmIh}[1]{\ul{\bmI}_{h}^{#1}}
\newcommand{\bmIhk}{\bmIh{k}}

\def\a{\rma}
\newcommand{\aT}{\a_T}
\newcommand{\sT}{\rms_T}
\newcommand{\ah}{\a_h}

\newcommand{\dT}{\rmd_T}
\def\dh{\rmd_h}

\newcommand{\tT}{\rmt_T}
\def\th{\rmt_h}

\newcommand{\wT}{w_T}
\newcommand{\wh}{w_h}

\newcommand{\bmvT}{\bmv_T}
\newcommand{\bmvh}{\bmv_h}
\newcommand{\bmvF}{\bmv_F}
\newcommand{\bmvFT}{\bmv_{\FT}}

\newcommand{\bmuT}{\bmu_T}
\newcommand{\bmuh}{\bmu_h}

\newcommand{\bmbh}{\bmb_h}

\newcommand{\bmwT}{\bmw_T}
\newcommand{\bmwh}{\bmw_h}

\newcommand{\bmwFT}{\bmw_{\FT}}

\newcommand{\bmzT}{\bmz_T}
\newcommand{\bmzh}{\bmz_h}
\newcommand{\bmzF}{\bmz_F}
\newcommand{\bmzFT}{\bmz_{\FT}}

\newcommand{\ulbmvT}{\ulbm{v}_T}
\newcommand{\ulbmuT}{\ulbm{u}_T}
\newcommand{\ulbmvh}{\ulbm{v}_h}
\newcommand{\ulbmuh}{\ulbm{u}_h}

\newcommand{\ulbmwT}{\ulbm{w}_T}
\newcommand{\ulbmzT}{\ulbm{z}_T}

\newcommand{\ulbmbT}{\ulbm{b}_T}
\newcommand{\ulbmwh}{\ulbm{w}_h}
\newcommand{\ulbmzh}{\ulbm{z}_h}

\newcommand{\ulbmbh}{\ulbm{b}_h}

\newcommand{\logLogSlopeTriangle}[5]
{
	\pgfplotsextra
	{
		\pgfkeysgetvalue{/pgfplots/xmin}{\xmin}
		\pgfkeysgetvalue{/pgfplots/xmax}{\xmax}
		\pgfkeysgetvalue{/pgfplots/ymin}{\ymin}
		\pgfkeysgetvalue{/pgfplots/ymax}{\ymax}
		
		\pgfmathsetmacro{\xArel}{#1}
		\pgfmathsetmacro{\yArel}{#3}
		\pgfmathsetmacro{\xBrel}{#1-#2}
		\pgfmathsetmacro{\yBrel}{\yArel}
		\pgfmathsetmacro{\xCrel}{\xArel}
		
		\pgfmathsetmacro{\lnxB}{\xmin*(1-(#1-#2))+\xmax*(#1-#2)} 
		\pgfmathsetmacro{\lnxA}{\xmin*(1-#1)+\xmax*#1} 
		\pgfmathsetmacro{\lnyA}{\ymin*(1-#3)+\ymax*#3} 
		\pgfmathsetmacro{\lnyC}{\lnyA+#4*(\lnxA-\lnxB)}
		\pgfmathsetmacro{\yCrel}{\lnyC-\ymin)/(\ymax-\ymin)}
		
		\coordinate (A) at (rel axis cs:\xArel,\yArel);
		\coordinate (B) at (rel axis cs:\xBrel,\yBrel);
		\coordinate (C) at (rel axis cs:\xCrel,\yCrel);
		
		\draw[#5]   (A)-- node[pos=0.5,anchor=north] {\scriptsize{1}}
		(B)-- 
		(C)-- node[pos=0.,anchor=west] {\scriptsize{#4}} 
		cycle;
	}
}

\begin{document}

\title{A hybrid high-order scheme for the stationary, incompressible magnetohydrodynamics equations}
 \author[1]{J\'er\^ome Droniou}
\author[2]{Liam Yemm}
 \affil[1]{School of Mathematics, Monash University, Melbourne, Australia, \email{jerome.droniou@monash.edu}}
\affil[2]{School of Mathematics, Monash University, Melbourne, Australia, \email{liam.yemm@monash.edu}}
\date{}
\maketitle

\begin{abstract}
  We propose and analyse a hybrid high-order (HHO) scheme for the stationary incompressible magnetohydrodynamics equations.
  The scheme has an arbitrary order of accuracy and is applicable on generic polyhedral meshes. For sources that are small
  enough, we prove error estimates in energy norm for the velocity and magnetic field, and $L^2$-norm for the pressure;
  these estimates are fully robust with respect to small faces, and of optimal order with respect to the mesh size.
  Using compactness techniques, we also prove that the scheme converges to a solution of the continuous problem, irrespective
  of the source being small or large. Finally, we illustrate our theoretical results through 3D numerical tests on tetrahedral
  and Voronoi mesh families.
\end{abstract}


\section{Introduction}\label{sec:introduction}

The theory of magnetohydrodynamics (MHD) models the interaction between magnetic fields and the motion of electrically conducting fluids \cite{davidson:2002:introduction}. Such models are ubiquitous in astrophysical systems \cite{goedbloed.poedts:2004:principles,goossens:2003:introduction} including the dynamics of interstellar dust clouds \cite{zweibel:1999:magnetohydrodynamics,shukla.rahman:1996:magnetohydrodynamics} and the dynamo effect responsible for the generation of planetary magnetic fields (including that of the earth) \cite{busse:1978:magnetohydrodynamics,glatzmaier.roberts:1995:dynamo,krause.radler:2016:mean}. One of the most heavily researched astrophysical applications of MHD is in solar physics \cite{priest:2014:magnetohydrodynamics}, for-instance the formation of stars \cite{schober.schleicher.ea:2012:small,machida.matsumoto.ea:2008:magnetohydrodynamics}, the coronal heating problem \cite{walsh.ireland:2003:heating,alfven.lindblad:1947:granulation,aschwanden:2006:physics}, the behaviour of the solar magnetic field \cite{charbonneau:2014:solar,krause.radler:2016:mean}, and the modelling of solar flares \cite{shibata:2011:solar,priest:1981:solar}. MHD is regularly applied to the physics of nuclear reactors, such as in the modelling of plasma dynamics in fusion reactors \cite{miyamoto:1980:plasma}, or that of liquid metal coolants in both fission \cite{forsberg:2005:advanced} and fusion \cite{tabares:2015:present} reactors. Magnetohydrodynamical theory is also increasingly being applied to industry \cite{davidson:1999:magnetohydrodynamics, al-habahbeh.al-saqqa.ea:2016:review} where magnetic fields are used to heat, pump, and control the flow of liquid metals.

Early numerical analysis of MHD began with the seminal paper \cite{gunzburger.meir.ea:1991:existence}, which proves existence and convergence of a discrete solution to the stationary, incompressible MHD equations in convex domains with $C^{1,1}$ boundary. The paper considers $\inf$-$\sup$ stable mixed elements for the kinetic variables, and $H^1$-conforming approximations of the magnetic unknowns. A more appropriate functional setting for the magnetic variables in more general domains is that of $H(\CURL;\Omega)$ \cite[cf.][]{costabel.dauge:2002:weighted}. Such a setting is provided by \cite{schotzau:2004:mixed}, which proposes a mixed finite element approximation shown to converge optimally in non-convex domains under a small data assumption. Convergence of a stabilised finite element method (FEM) for linearised MHD is shown in \cite{gerbeau:2000:stabilized}, where the nonlinear problem is approximated by a coupled Picard algorithm which requires resolving the linear case at each iteration. A convergent method for the non-stationary problem is provided by \cite{prohl:2008:convergent}. The early work of \cite{meir.schmidt:1999:analysis} provides analysis and numerical theory of MHD with non-ideal boundary, which resolves some of the many difficulties of modelling realistic magnetic flow problems. A review of the various numerical methods at the time for the discretisation of three-dimensional incompressible MHD is given in \cite{salah.soulaimani.ea:2001:finite}. A more detailed description of the known results for MHD is provided by \cite{gerbeau.le-bris.ea:2006:mathematical}. Due to the saddle-point structure of both the fluid and magnetic equations, many of the aforementioned conforming methods require elements satisfying an $\inf$-$\sup$ condition. This becomes problematic as the Reynolds and magnetic Reynolds number become too large. A stabilised finite element discretisation is proposed in \cite{badia.codina.ea:2013:unconditionally} and shown to converge unconditionally to the physical solution -- even in the presence of singular solutions. One of the earliest works focused on conserving the divergence-free constraints on the magnetic field is found in \cite{salah.soulaimani.eq:1999:conservative}, which provides a stabilised mixed method for the non-stationary problem. However, only the magnetic equations are discretised, and the fluid velocity is assumed to be prescribed. More recently, a stabilised FEM for stationary MHD is designed in \cite{hu.ma.ea:2017:stable} which preserves the divergence free constraints of both the velocity and magnetic fields at the discrete level. We also make note of the recent publications \cite{hu.qiu.ea:2020:convergence,qiu.shi:2020:analysis,he.dong.ea:2022:uniform}. A non-conforming approximation of the linearised model is proposed in \cite{houston.schotzau.ea:2009:mixed} using a mixed discontinuous Galerkin (DG) approach. Following this work, a DG approximation of the non-linear model is provided by \cite{qiu.shi:2020:mixed} which conserves the divergence free conditions. The DG approach has the drawback of leading to a large number of degrees of freedom. This issue can be alleviated, as shown in \cite{qiu.shi:2020:mixed}, through the use of a hybridisable discontinuous Galerikin (HDG) approximation, for which element degrees of freedom can be locally eliminated through static condensation. The recent article \cite{gleason.peters.ea:2022:divergence} devises a HDG method which locally conserves the divergence free constraints of both the fluid and magnetic field, and satisfies a global momentum balance. For standard diffusion models, HDG and Hybrid-High Order methods share many similarities \cite{cockburn.di-pietro.ea:2016:bridging}; however, the HDG method of \cite{qiu.shi:2020:mixed,gleason.peters.ea:2022:divergence} use hybrid spaces also for the fluid and magnetic pressure, which is not the usual choice in HHO for Navier--Stokes, and leads to many more degrees of freedom -- as expected to achieve improved conservation of divergence (see discussion in Remark \ref{rem:hybrid.spaces} below).

Developed in \cite{di-pietro.ern.ea:2014:arbitrary, di-pietro.ern:2015:hybrid}, hybrid high-order (HHO) schemes are modern polytopal methods for the approximation of elliptic PDEs. A key aspect of HHO is its applicability to generic meshes with arbitrarily shaped elements. Additionally, HHO methods are of arbitrary order, dimension independent, and are built on polynomial reconstructions that account for the local physics and enable robustness with respect to the model's parameters. Hybrid high-order methods are also amenable to static condensation of the system matrix which drastically reduces the number of globally coupled degrees of freedom. A thorough review of the analysis and application of HHO methods is given in the monograph \cite{di-pietro.droniou:2020:hybrid}. There have been numerous studies of hybrid high-order discretisations of the Stokes \cite{di-pietro.ern.ea:2016:discontinuous, botti.di-pietro:2022:p-multilevel} and Navier-Stokes \cite{botti.di-pietro.ea:2019.hybrid,di-pietro.krell:2018:hybrid} equations. By considering a hybrid pressure space, an HHO discretisation of the Navier-Stokes problem can be devised which locally preserves the conservation of mass of the fluid, and even exhibit robustness in the incompressible Euler limit \cite{botti.massa:2022:hho}. The recent article \cite{chave.di-pietro.lemaire:2022:discrete} devises a HHO method for a magnetostatics problem, albeit without fully exploiting the principle of high-order reconstruction of HHO methods for the curl operator. However, to the best of our knowledge, no hybrid high-order schemes for MHD models exist.

Closely related to the HHO method are the virtual element methods (VEM) \cite{beirao-da-veiga.brezzi.ea:2013:basic,brenner.sung:2018:virtual,ayuso-de-dios.lipnikov.ea:2016:nonconforming}, which are also in principle arbitrary-order polytopal methods. A VEM Stokes complex is developed in \cite{BDV:3D:ST}, and used in the recent work \cite{vem:mhd:22} to design a VEM scheme for MHD equations. However, only the lowest order VEM is considered and analysed in \cite{vem:mhd:22}.

This paper designs and analyses an HHO discretisation of the stationary, incompressible MHD equations. By considering the problem posed in a polyhedral domain we are able to consider a weak formulation where both the fluid and magnetic unknowns are $H^1$ and the main differential operator is the Laplacian, and exploit certain symmetries of the model. Compared to the existing literature, our main contributions are:
\begin{itemize}
\item arbitrary-order scheme for the full MHD equations, that is applicable on generic polyhedral meshes,
\item optimal-order error estimates, in the case of small source terms, that are fully robust with respect to small faces,
\item convergence analysis without any assumption on the magnitude of the source terms,
\item 3D tests with various polynomial degrees.
\end{itemize}
The paper is organised as follows. In the next subsection, we briefly present the governing equations and recast them into a suitable form for an HHO discretisation. The scheme itself is presented in Section \ref{sec:scheme}, with the main convergence results in Section \ref{sec:main.results} (including uniqueness of the solution and error estimates for small source terms, and existence of the solution and convergence irrespective of the source term).
The analysis of the scheme is carried out in Section \ref{sec:proofs}. We then present 3D simulations on tetrahedral and Voronoi (polyhedral) meshes in Section \ref{sec:numerical}, for polynomial degrees up to two (rates of convergence up to three).
A conclusion is presented in Section \ref{sec:conclusion}.

\subsection{Governing equations}

The mathematics of MHD combines the Navier--Stokes equations governing the motion of a viscous fluid, and Maxwell’s equations and the Lorentz force at the core of electromagnetism. A number of key assumptions are made in deriving the central model \eqref{eq:mhd} used throughout this paper. We assume that the fluid satisfies a continuum assumption so that any discrete behaviour of particles may be ignored, and that typical length scales are much greater than the Debye length so that the net charge of any fluid element can be assumed to be zero. We consider here the case of an incompressible, Newtonian fluid. This assumes a sufficiently small Mach number, and a linear relation between the stress and strain tensors. We also assume that the fluid velocity is sufficiently smaller than the speed of light to ignore any relativistic effects, and that the relationship between the fluid velocity and magnetic field is governed by Ohm's law. All the assumptions mentioned above fall in one of the following four categories: the typical length scale is sufficiently large, the fluid velocity is sufficiently small, the viscosity is sufficiently small, and the electrical conductivity is sufficiently large.

Given a medium with constant viscosity $\nu$ and constant magnetic permeability $\mu$, the incompressible magnetohydrodynamic equations read:
\begin{align*}
	- \nu_{k}\Delta \bmu + (\bmu \cdot \nabla)\bmu + \nabla \frac{p}{\rho} - (\nabla \times \bmb) \times \bmb  \eq -\frac{\partial \bmu}{\partial t} + \bmf, \\
	\DIV \bmu \eq 0, \\
	\nu_m \nabla \times \nabla \times \bmb - \nabla \times(\bmu \times \bmb) \eq -\frac{\partial \bmb}{\partial t}, \\
	\DIV \bmb \eq 0,  
\end{align*}
where $p$ and $\bmu$ denote the fluid pressure and velocity respectively, $\bmf$ the external body force per unit mass, $\rho$ is the mass density, $\nu_k=\frac{\nu}{\rho}$ and $\nu_m=\frac{1}{\mu \sigma}$ (where $\sigma$ is the electrical conductivity) are constants representing kinematic viscosity and magnetic diffusivity respectively, and $\bmb$ represents (after the re-scaling $\bmb \mapsto \frac{1}{\sqrt{\mu \rho}} \bmb$) the magnetic field.

Due to the zero divergence conditions on both $\bmb$ and $\bmu$, the following identities hold:
\begin{align*}
\nabla \times \nabla \times \bmb \eq - \Delta \bmb, \\
- \nabla \times(\bmu \times \bmb) \eq (\bmu \cdot \nabla) \bmb - (\bmb \cdot \nabla) \bmu.
\end{align*}
We also note that
\[	
(\nabla \times \bmb) \times \bmb = (\bmb \cdot \nabla) \bmb - \nabla \left(\frac12 \bmb \cdot \bmb\right).
\]
For the sake of well-posedness, a Lagrange multiplier $r$ for the divergence free constraint on $\bmb$ is introduced to the magnetic equations. Thus, we consider the steady state problem: given a bounded Lipschitz domain $\Omega\subset\R^3$, find the fluid velocity $\bmu$, the magnetic field $\bmb$, and scalars $q$ and $r$ such that
\begin{subequations}\label{eq:mhd}
\begin{align}	
- \nu_k \Delta \bmu + (\bmu \cdot \nabla)\bmu - (\bmb \cdot \nabla)\bmb + \nabla q \eq \bmf, \label{eq:mhd.fluid}\\
- \nu_m \Delta \bmb + (\bmu \cdot \nabla) \bmb - (\bmb \cdot \nabla) \bmu + \nabla r \eq \bmg, \label{eq:mhd.magnetic}\\
\DIV \bmu \eq 0, \label{eq:mhd.fluid.div}\\
\DIV \bmb \eq 0, \label{eq:mhd.magnetic.div}
\end{align}
for some source terms $\bmf, \bmg\in \LTWO(\Omega)^3$, with $\DIV\bmg=0$. We consider the boundary conditions
\begin{alignat}{2}
\bmu \eq \bm{0} \qquad &\textrm{on }& \partial\Omega, \label{eq:mhd.bc.u}\\
\bmn \times (\nabla \times \bmb) \eq \bm{0} \qquad &\textrm{on }& \partial\Omega,\label{eq:mhd.bc.curlb} \\
\bmb \cdot \bmn \eq 0 \qquad &\textrm{on }& \partial\Omega, \label{eq:mhd.bc.b}\\
\int_{\Omega} r =\int_{\Omega} q\eq 0.\label{eq:mhd.bc.rq}
\end{alignat}
\end{subequations} 

The regularity of weak solutions of such a formulation (with the Laplacian acting as the key differential operator in the magnetic equations) is investigated in \cite{chen.miao.eq:2008:regularity} for the non-stationary problem in the whole space $\R^3$.

\begin{remark}[Boundary conditions]
	The boundary condition \eqref{eq:mhd.bc.u} is the classical no-slip condition, and \eqref{eq:mhd.bc.curlb} represents a perfectly conducting wall. The condition \eqref{eq:mhd.bc.b} ensures continuity of the normal component of the magnetic field at the boundary and should be more accurately described by $\bmb\cdot\bmn = b_n$, for some prescribed $b_n$ representing the normal component of $\bmb$ outside the boundary. However, to ease the exposition we consider here homogeneous conditions. A detailed discussion of boundary conditions for MHD can be found in \cite{potherat.sommeria.ea:2002:effective}.
\end{remark}

\subsection{A weak formulation}

To justify the treatment of boundary conditions in the weak formulation of \eqref{eq:mhd}, we need the following lemma.

\begin{lemma}\label{lem:boundary.term}
	Let $\bmv:\Omega\to\R^3$ be a piecewise $H^2$ vector field on $\Omega$ and $\bmw\in L^2(\partial\Omega;\R^3)$. We assume that the outer normal $\bmn:\partial\Omega\to\R^3$ to $\Omega$ is piecewise $C^2$, and that
	\begin{equation}
		\bmw \cdot \bmn = \bmv \cdot \bmn = 0 \qquad \textrm{a.e.~on } \partial\Omega. \label{eq:vw.tangent}
	\end{equation}
	Then,
	\[
	\bmw \cdot (\bmn \cdot \nabla) \bmv = -\bmkappa\bmw\cdot\bmv - \bmw\cdot[\bmn\times(\nabla\times \bmv)]\quad\mbox{ a.e.~on $\partial\Omega$},
	\]
	where $\boldsymbol{\kappa}=(\partial_j n_i)_{i,j}$ is the extrinsic curvature of $\partial \Omega$.
	As a consequence, if $\Omega$ is polyhedral then 
	\begin{equation}
	\bmw \cdot (\bmn \cdot \nabla) \bmv =- \bmw\cdot[\bmn\times(\nabla\times \bmv)]\quad\mbox{ a.e.~on $\partial\Omega$}.
	\label{eq:bdry.cancellation.polyhedral}
	\end{equation}
\end{lemma}

\begin{proof}
	We have the general formula 
	\begin{equation}\label{eq:form.curl}
		(\bmn\cdot\nabla)\bmv=\nabla_{\bmv}(\bmn\cdot\bmv)-\bmn\times(\nabla\times \bmv),
	\end{equation}
	where the gradient $\nabla_{\bmv}(\bmn\cdot\bmv)$ is computed assuming that $\bmn$ is constant -- that is, the $i$-th component of this vector is 
	\[
		\sum_{j=1}^3 n_j\partial_i v_j=\partial_i(\bmn\cdot\bmv) - \sum_j(\partial_i n_j)v_j=\partial_i(\bmn\cdot\bmv)-(\bmkappa^\intercal\bmv)_i.
	\]
	Hence, $\bmw\cdot\nabla_{\bmv}(\bmn\cdot\bmv)=\bmw \cdot\nabla(\bmn\cdot\bmv) - \bmkappa\bmw\cdot\bmv$.
	By \eqref{eq:vw.tangent}, the vector field $\bmw$ is tangent to $\partial\Omega$ and thus $\bmw\cdot\nabla$ is a tangential derivative. Since $\bmn\cdot\bmv=0$ on $\partial\Omega$, its tangential derivative also vanishes and thus $\bmw\cdot\nabla_{\bmv}(\bmn\cdot\bmv)=-\bmkappa\bmw\cdot\bmv$. Combined with \eqref{eq:form.curl}, this concludes the proof of the first part of the lemma.
	
	If $\Omega$ is polyhedral, then its boundary is smooth on all its faces (so, a.e.), on which $\boldsymbol{\kappa}=\bm{0}$. This proves the second part of the lemma.
\end{proof}

In the following, we consider a polyhedral domain $\Omega$. Assuming that the vector fields $\bmv,\bmw$ are sufficiently smooth and satisfy $\bmv\cdot\bmn=\bmw\cdot\bmn=0$ on $\bdry \Omega$ and $\bmw|_{\partial\Omega}\in L^2(\partial\Omega;\R^3)$, and  noticing that $\sum_{i=1}^3 w_i \nabla v_i \cdot \nor=
\bmw \cdot (\bmn \cdot \nabla) \bmv$, we infer from Lemma \ref{lem:boundary.term} the following integration-by-parts formula:
\begin{equation}\label{eq:ibp.boundary}
-\int_\Omega \bmw \cdot \Delta \bmv = \sum_{i=1}^3\left( \int_\Omega \nabla w_i \cdot \nabla v_i - \int_{\partial \Omega} w_i \nabla v_i \cdot \nor\right) \\
= \int_\Omega \nabla \bmw : \nabla \bmv + \int_{\partial \Omega} \bmw\cdot[\bmn\times(\nabla\times \bmv)].
\end{equation}

Let us define the spaces
\begin{align*}	
	\bmU \eq \HONEzr(\Omega)^3, \\
	\bmB \eq \{\bmb \in \HONE(\Omega)^3: \bmb \cdot \bmn = 0 \textrm{ on } \partial\Omega\}, \\
	P \eq \left\{p \in \LTWO(\Omega): \int_{\Omega}p = 0\right\}.
\end{align*}
By multiplying \eqref{eq:mhd.fluid} by $\bmv\in \bmU$ and \eqref{eq:mhd.magnetic} by $\bmw\in \bmB$, integrating on $\Omega$ and using \eqref{eq:ibp.boundary}, we see that the strong form \eqref{eq:mhd} of the MHD equations is formally equivalent to the following weak formulation: find $\bmu \in \bmU$, $\bmb \in \bmB$, $q, r \in P$ such that
\begin{subequations}\label{eq:variational}
\begin{alignat}{2}
 \nu_k \a(\bmu, \bmv) + \rmt(\bmu, \bmu, \bmv) - \rmt(\bmb, \bmb, \bmv) + \rmd(\bmv, q) \eq \brac[\Omega]{\bmf, \bmv} \qquad &\forall& \bmv\in \bmU, \label{eq:variational.fluid}\\
 \nu_m \a(\bmb, \bmw) + \rmt(\bmu, \bmb, \bmw) - \rmt(\bmb, \bmu, \bmw) + \rmd(\bmw, r) \eq \brac[\Omega]{\bmg, \bmw} \qquad &\forall& \bmw\in \bmB, \label{eq:variational.magnetic}\\
 -\rmd(\bmu, s) \eq 0 \qquad &\forall& s\in P, \label{eq:variational.fluid.div}\\
 -\rmd(\bmb, z) \eq 0 \qquad &\forall& z\in P, \label{eq:variational.magnetic.div}
\end{alignat}
\end{subequations}
with bilinear forms $\a : \HONE(\Omega)^3 \times \HONE(\Omega)^3 \to \R$ and $\rmd : \HONE(\Omega)^3 \times \LTWO(\Omega) \to \R$ defined by
\begin{equation*}
	\rma(\bmv, \bmw) \defeq \int_{\Omega} \nabla \bmv : \nabla \bmw,\qquad \rmd(\bmv, s) \defeq - \int_{\Omega} (\nabla \cdot \bmv) s,
\end{equation*}
and trilinear form $\rmt:\HONE(\Omega)^3 \times \HONE(\Omega)^3 \times \HONE(\Omega)^3 \to \R$ such that
\begin{equation}\label{eq:t.def}
	\rmt(\bmv, \bmw, \bmz) \defeq \int_{\Omega}(\bmv \cdot \nabla)\bmw \cdot \bmz.
\end{equation}

\begin{remark}[Regularity of $\bmb$]	
	By applying the integration by parts formula \eqref{eq:ibp.boundary} to derive equation \eqref{eq:variational.magnetic} we have implicitly assumed that $\bmb$ is smooth enough. The resulting formulation also seeks for $\bmb\in\HONE(\Omega)^3$, which may be stronger than the $H(\CURL;\Omega)$ regularity that is expected on the magnetic field in some situations. To cover such cases, a different HHO approach needs to be considered, such as the one in \cite{chave.di-pietro.lemaire:2022:discrete}.
\end{remark}

\section{Hybrid high-order scheme for the MHD model}\label{sec:scheme}

We consider polyhedral meshes as in \cite[Definition 1.4]{di-pietro.droniou:2020:hybrid}.
Fixing a countable set of mesh sizes \(\calH\subset(0, \infty)\) with a unique cluster point at \(0\), for each $h\in\calH$ we consider a mesh \(\Mh=(\Th, \Fh)\) of the domain $\Omega$ as follows: $\Th$ is a collection of disjoint open polyhedra (the elements) such that $\ol{\Omega}=\cup_{T\in\Th}\ol{T}$, and $\Fh$ is a collection of disjoint planar sets (the faces) such that $\cup_{T\in\Th}\bdryT=\cup_{F\in\Fh}\ol{F}$.

We shall also collect the boundary faces \(F\subset\partial \Omega\) in the set \(\Fhb\), and the interior faces in the set $\Fhi = \Fh \backslash \Fhb$. For \(X=T\in\Th\) or \(X=F\in\Fh\), \(\hX\) denotes the diameter of \(X\). 
The parameter \(h\) is related to the elements diameters by \(h\defeq\max_{T\in\Th}\hT\). For each \(T\in\Th\), the set \(\FT:=\{F\in\Fh:F\subset \bdryT\}\) denotes the collection of faces contained in the boundary of $T$. Similarly, the one or two elements attached to a face $F\in\Fh$ is denoted by $\Th[F] \defeq \{T\in\Th : F\subset\bdryT\}$. The (constant) unit normal to \(F\in\FT\) pointing outside \(T\) is denoted by \(\norTF\), and \(\norT:\bdryT\to\R^3\) is the piecewise constant outer unit normal defined by \((\norT)|_F=\norTF\).

We make the following assumption on the mesh which ensures that boundary trace inequalities, and Lebesgue and Sobolev embeddings hold independent of the parameter $h$.

\begin{assumption}[Regular mesh sequence]\label{assum:star.shaped}	
	There exists a constant \(\varrho>0\) such that, for each \(h\in\calH\), each \(T\in\Th\) is star-shaped with respect to a ball of radius $r_T \ge \varrho^{-1} \hT$.
\end{assumption}

\begin{remark}[About the mesh regularity property]
	It has recently been proved in \cite{droniou.yemm:2021:robust} that HHO methods for elliptic equations are robust and convergent under weaker mesh regularity assumptions than Assumption \ref{assum:star.shaped}. The analysis carried out here would also hold under the weaker assumption introduced in this reference. We adopt here the slightly stronger Assumption \ref{assum:star.shaped} (which still allows for arbitrarily shaped and sized faces, on the contrary to standard mesh regularity assumptions for HHO methods -- see e.g.~\cite[Definition 1.9]{di-pietro.droniou:2020:hybrid}) to simplify the presentation, mostly the proof of the discrete Poincar\'e--Sobolev--Wirtinger inequality in Lemma \ref{lem:sob.poincare}.
\end{remark}

From hereon, we shall denote \(X\lesssim Y\) to mean \(X \le CY\) where \(C\) is a constant independent of the quantities $X$ and $Y$, of the mesh size $h$, and of the model data $(\nu_k,\nu_m,\bmf,\bmg)$.

\subsection{Local construction}

In this section we construct local approximations of each of the bilinear forms $\a$ and $\rmd$, and the trilinear form $\rmt$ on each element $T\in\Th$. The local space of functions on each element is given by a couple \((u,v)\) where \(u\) is a polynomial on the element and \(v\) is a piecewise discontinuous polynomial function on the boundary. For each of the differential operators appearing in the continuous formulation, we define discrete analogues acting on the local space. We can then define local discrete forms $\aT$, $\dT$, $\tT$ which mimic their continuous counterparts. These constructions follow \cite[Chapter 9]{di-pietro.droniou:2020:hybrid} but we recall them for the sake of legibility.

Let \(X=T\in\Th\) or \(X=F\in\Fh\) be an element or a face in a mesh \(\Mh\), given an integer $k\ge 0$, we denote by \(\POLY{k}(X)\) the set of \(d_X\)-variate polynomials of total degree \(\le k\) on \(X\), where \(d_X\) is the dimension of \(X\). The space of piecewise discontinuous polynomial functions on an element boundary is given by
\begin{equation*}
	\POLY{k}(\Fh[T]) \defeq \{v\in L^1(\bdryT):v|_F\in\POLY{k}(F)\quad\forall F\in\Fh[T]\}.
\end{equation*}
On each element $T\in\Th$, the local space of unknowns is defined as follows:
\begin{equation*}
	\bmUTk \defeq \POLY{k}(\T)^3 \times \POLY{k}(\FT)^3.
\end{equation*}
We also endow the space $\bmUTk$ with the $H^1$-like seminorm $\norm[1, T]{\cdot}:\bmUTk\to\R$ defined for all $\ulbmvT\in\bmUTk$ via
\begin{equation}\label{eq:local.norm.def}
	\norm[1, T]{\ulbmvT}^2 \defeq \norm[\T]{\nabla \bmvT}^2 + \hT^{-1}\norm[\bdryT]{\bmvFT - \bmvT}^2.
\end{equation}
The local interpolator $\bmITk:\HONE(T)^3\to\bmUTk$ is given by
\[
	\bmITk \bmv = (\bmpiTzr{k}\bmv, \bmpiFTzr{k}\bmv)\qquad\forall \bmv\in\HONE(T)^3
\]
where $\bmpiTzr{k}$ and $\bmpiFTzr{k}$ denote the $\LTWO$-orthogonal projectors on the spaces $\POLY{k}(\T)^3$ and $\POLY{k}(\FT)^3$, respectively.

\subsubsection{Discretisation of the bilinear forms}

The local reconstruction $\rT{k+1}:\bmUTk\to\POLY{k+1}(T)^3$ is defined such that for every $\ulbmvT = (\bmvT, \bmvFT) \in \bmUTk$,
\begin{align}\label{eq:rT.def}
	\int_T \nabla \rT{k+1}\ulbmvT : \nabla \bmw ={}& \int_T \nabla\bmvT : \nabla \bmw + \int_{\partial T} (\bmvFT-\bmvT) \cdot (\norT \cdot \nabla) \bmw \qquad \forall \bmw \in \POLY{k+1}(T)^3,\\
\nonumber
	\int_T (\rT{k+1}\ulbmvT - \bmvT) ={}& \bm{0}.
\end{align}
We also define the divergence reconstruction $\DT{k}:\bmUTk\to\POLY{k}(T)$ to satisfy
\begin{equation}\label{eq:DT.def}
	\int_T \DT{k} \ulbmvT q = -\int_T \bmvT \cdot \nabla q + \int_{\partial T} (\bmvFT \cdot \norT) q \qquad \forall q \in \POLY{k}(T).
\end{equation}
The potential and divergence reconstructions satisfy \cite[Eqs.~(8.17) and (8.21)]{di-pietro.droniou:2020:hybrid} the following properties:
\begin{equation}\label{eq:rT.and.DT.commutation}
	\rT{k+1} \bmITk \bmv = \bmpiTe{k+1} \bmv \quad\textrm{and}\quad \DT{k}\bmITk \bmv = \piTzr{k} (\nabla \cdot \bmv)
\end{equation}
where $\bmpiTe{k+1}$ is the elliptic projector \cite[Definition 1.39]{di-pietro.droniou:2020:hybrid}, which enjoys optimal volumetric approximation properties \cite[Theorem 1.48]{di-pietro.droniou:2020:hybrid} (including for the weaker mesh assumptions we consider, see \cite[Lemma 5]{droniou.yemm:2021:robust}).

The continuous bilinear form $\rma$ is approximated by the discrete form $\aT:\bmUTk\times\bmUTk\to\R$ defined via
\begin{equation*}
	\aT(\ulbmuT, \ulbmvT) \defeq \int_T\nabla\rT{k+1}\ulbmuT : \nabla\rT{k+1}\ulbmvT + \sT(\ulbmuT, \ulbmvT)
\end{equation*}
where $\sT:\bmUTk\times\bmUTk\to\R$ is a symmetric positive semi-definite stabilisation bilinear form that satisfies boundedness and consistency properties:
\begin{align}\label{eq:aT.norm.equivalence}
	\norm[1, T]{\ulbmvT}^2 \lesssim{}& \aT(\ulbmvT, \ulbmvT) \lesssim \norm[1, T]{\ulbmvT}^2\qquad\forall\ulbmvT\in\bmUTk,\\
\label{eq:sT.consistency}
	\sT(\bmIT{k}\bmv, \bmIT{k}\bmv) \lesssim{}& \big[\hT^{k+1} \seminorm[H^{k+2}(T)]{\bmv}\big]^2\qquad\forall \bmv\in H^{k+2}(T)^3.
\end{align}
Examples of such stabilisation forms are given in \cite[Section 4]{droniou.yemm:2021:robust}.

We approximate $\rmd$ by the discrete form $\dT:\bmUTk\times\POLY{k}(T)\to\R$ defined via
\begin{equation*}
	\dT(\ulbmvT, q_T) = -\int_T \DT{k}\ulbmvT q_T.
\end{equation*}

\subsubsection{Discretisation of the trilinear form}

Consider, for $l\ge 0$, the gradient reconstruction $\GT{l}:\bmUTk\to\POLY{l}(T)^{3\times 3}$ defined to satisfy for each $\ulbmvT\in\bmUTk$
\begin{equation*}
	\int_T \GT{l}\ulbmvT : \bmtau = -\int_T\bmvT \cdot (\nabla \cdot \bmtau) + \int_{\partial T} \bmvFT \cdot (\bmtau \norT) \qquad\forall\bmtau\in\POLY{l}(T)^{3\times 3}.
\end{equation*}
For any $\bmv = (v_i)_{i=1,2,3} \in \LTWO(T)^3$, we then define the convective derivative $\bmv \cdot \GT{2k}:\bmUTk\to\LTWO(T)^3$ via
\[
	\SqBrac[i]{(\bmv \cdot \GT{2k})\ulbmwT} = \sum_{j=1}^3 v_j (\GT{2k}\ulbmwT)_{ij}  \qquad \forall \ulbmwT\in\bmUTk.
\]
By \cite[Remark 9.16]{di-pietro.droniou:2020:hybrid}, for any $\ulbmvT, \ulbmwT, \ulbmzT\in\bmUTk$ it holds that
\begin{equation}\label{eq:GT.expanded}
	\int_T (\bmvT \cdot \GT{2k})\ulbmwT\cdot\bmzT = \int_T (\bmvT \cdot \nabla)\bmwT\cdot\bmzT + \int_{\bdryT} (\bmvT \cdot \norT)(\bmwFT - \bmwT) \cdot \bmzT.
\end{equation}
The discrete trilinear form $\rmt_T:\bmUTk\times\bmUTk\times\bmUTk\to\R$ is defined as
\begin{equation}\label{eq:tT.def}
	\tT(\ulbmvT, \ulbmwT, \ulbmzT) \defeq \frac12 \BRAC{\int_T (\bmvT \cdot \GT{2k})\ulbmwT\cdot\bmzT- \int_T(\bmvT \cdot \GT{2k})\ulbmzT\cdot\bmwT}.
\end{equation}

\subsection{Discrete problem}

The global discrete space is defined as
\[
	\bmUhk \defeq \Big\{\ulbmvh=((\bmvT)_{T\in\Th},(\bmvF)_{F\in\Fh})\,:\,\bmvT\in\POLY{k}(T)^3\quad\forall T\in\Th\,,
	\bmvF\in\POLY{k}(F)^3\quad\forall F\in\Fh \Big\},
\]
and is equipped with the global seminorm $\norm[1, h]{\cdot}:\bmUhk\to\R$ and interpolator $\bmIhk:\HONE(\Omega)^3\to\bmUhk$ defined via
\begin{align*}
	\norm[1, h]{\ulbmvh}^2 \defeq{}& \sum_{T\in\Th} \norm[1, T]{\ulbmvT}^2\qquad\forall \ulbmvh\in\bmUhk,\\
	\bmIhk \bmv \big|_T ={}& \bmITk \bmv\qquad\forall \bmv\in\HONE(\Omega)^3.
\end{align*}

Similarly, for all $\ulbmvh, \ulbmwh, \ulbmzh\in\bmUhk$ and $q_h\in\POLY{k}(\Th)$, we define
\begin{align}\label{eq:ah.and.dh.def}
	\ah(\ulbmvh, \ulbmwh) \defeq{}& \sum_{T\in\Th}\aT(\ulbmvT, \ulbmwT), \qquad \dh(\ulbmvh, q_h) \defeq \sum_{T\in\Th}\dT(\ulbmvT, q_T),\\
\label{eq:th.def}
	\th(\ulbmvh, \ulbmwh, \ulbmzh) \defeq{}& \sum_{T\in\Th}\tT(\ulbmvT, \ulbmwT, \ulbmzT).
\end{align}

The boundary conditions on $\bmu$ and $\bmb$ are accounted for in the following two homogeneous subspaces of $\bmUhk$:
\begin{align*}
	\bmUhkzr \deq \Big\{\ulbmvh=((\bmvT)_{T\in\Th},(\bmvF)_{F\in\Fh}) \in \bmUhk : \bmvF = \bm{0}\quad\forall F \subset \partial\Omega \Big\}, 
	\\
	\bmUhkn \deq \Big\{\ulbmvh=((\bmvT)_{T\in\Th},(\bmvF)_{F\in\Fh}) \in \bmUhk : \bmvF\cdot\norF = 0\quad\forall F \subset \partial\Omega \Big\}. 
\end{align*}
It is clear that
\begin{equation*}
	\bmUhkzr \subset \bmUhkn.
\end{equation*}
Defining the pressure space as
\begin{equation}\label{def:Pk0}
\POLY{k}_0(\Th)\defeq \left\{z\in \LTWO(\Omega)\,:\,z_{|T}\in\POLY{k}(T)\quad\forall T\in\Th\,,\quad\int_\Omega z=0\right\},
\end{equation}
we set the following global discrete space, written as a Cartesian product of each of the global spaces:
\[
	\bmX_h^k \defeq \bmUhkzr\times\bmUhkn\times\POLY{k}_0(\Th)\times\POLY{k}_0(\Th).
\]
For $\ulbmzh\in \bmUhkzr$ or $\ulbmzh\in \bmUhkn$, we denote by $\bmzh:\Omega\to\R^3$ the piecewise polynomial function defined by $(\bmzh)_{|T}=\bmzT$ for all $T\in\Th$.
The discrete problem then reads: find $(\ulbmuh, \ulbmbh, q_h, r_h) \in \bmX_h^k$ such that
\begin{subequations}\label{eq:discrete}
\begin{align}	
	\nu_k \ah(\ulbmuh, \ulbmvh) + \th(\ulbmuh, \ulbmuh, \ulbmvh) - \th(\ulbmbh, \ulbmbh, \ulbmvh) + \dh(\ulbmvh, q_h) \eq \brac[\Omega]{\bmf, \bmvh} \qquad \forall \ulbmvh\in\bmUhkzr, \label{eq:discrete.fluid}\\
	\nu_m \ah(\ulbmbh, \ulbmwh) + \th(\ulbmuh, \ulbmbh, \ulbmwh) - \th(\ulbmbh, \ulbmuh, \ulbmwh) + \dh(\ulbmwh, r_h) \eq \brac[\Omega]{\bmg, \bmwh} \qquad \forall \ulbmwh\in\bmUhkn, \label{eq:discrete.magnetic}\\
	-\dh(\ulbmuh, s_h) \eq 0 \qquad \forall s_h\in \POLY{k}_0(\Th), \label{eq:discrete.fluid.div}\\
	-\dh(\ulbmbh, z_h) \eq 0 \qquad \forall z_h\in \POLY{k}_0(\Th). \label{eq:discrete.magnetic.div}
\end{align}
\end{subequations}

\begin{remark}[Hybrid pressure spaces]\label{rem:hybrid.spaces}
Hybrid pressure spaces --that is, discrete pressure spaces with unknowns in the elements and on the faces-- have been considered for the hybridizable discontinuous Galerkin discretisation of the Navier--Stokes equations in \cite{rhebergen.wells:2018:hybridizable}, and for the HHO discretisation of Stokes and Navier--Stokes equations in \cite{botti.di-pietro:2022:p-multilevel,botti.massa:2022:hho} (which, for a given number of globally coupled unknowns, achieves a higher degree of accuracy than \cite{rhebergen.wells:2018:hybridizable} thanks to the HHO potential reconstruction). Using hybrid pressure spaces ensures that the element velocities are pointwise divergence-free and have continuous normal traces (that is, $\DIV \bmuT=0$ for each $T\in\Th$, and $\bmu_T\cdot\bmn_{TF}+\bmu_{T'}\cdot\bmn_{T'F}=0$ for all $F\in\FT\cap\mathcal F_{T'}$). 

A variation of the HHO scheme \eqref{eq:discrete} with hybrid pressure spaces could be designed following the approach of \cite{botti.di-pietro:2022:p-multilevel,botti.massa:2022:hho}, and would lead to a scheme for which element velocities and magnetic fields are pointwise divergence-free and with continuous normal traces. This would, however, substantially increase the total number of globally coupled unknowns (after static condensation): for \eqref{eq:discrete}, the globally coupled unknowns are the face velocities and magnetic fields, and one pressure unknown per element; for a hybrid pressure version, the globally coupled unknowns would be the face velocities and magnetic fields, but also all the face unknowns of $q,r$. For the MHD model, which essentially consists of two coupled Navier--Stokes problems, this can lead to quite an expensive system to solve. We however note that some form of (reconstructed) divergence-free property is obtained by \eqref{eq:discrete}, namely: $\DT{k} \ulbmuT=\DT{k} \ulbmbT=0$ for all $T\in\Th$.
\end{remark}

\subsection{Main results}\label{sec:main.results}

We cite in this section the main results on the HHO approximation of \eqref{eq:variational}. Their proof is given in Section \ref{sec:proofs}. These results cover the existence of a solution to the HHO scheme, its convergence under general data (which provides as a by-product the existence of a weak solution to the MHD model) and, under a smallness assumption on the data, the uniqueness of the solution and an error estimate.

\begin{theorem}[Existence of a discrete solution]\label{thm:existence}
	There exists at least one solution $(\ulbmuh, \ulbmbh, q_h, r_h) \in \bmX_h^k$ to equations \eqref{eq:discrete.fluid}--\eqref{eq:discrete.magnetic.div}. Moreover, any solution $(\ulbmuh, \ulbmbh, q_h, r_h)$ satisfies the a priori bounds 
  \begin{equation}
	\label{eq:discrete.apriori.bound}
	\begin{aligned}
	\Brac{\nu_k \norm[1, h]{\ulbmuh}^2 + \nu_m \norm[1, h]{\ulbmbh}^2}^\frac12 \les\max(\nu_k,\nu_m)^{-\frac12}\left(\norm[\Omega]{\bmf}^2 + \norm[\Omega]{\bmg}^2\right)^\frac12,\\
		\norm[\Omega]{q_h} + \norm[\Omega]{r_h} \le{}& C\Brac{\norm[\Omega]{\bmf}^2 + \norm[\Omega]{\bmg}^2}^\frac12 \left(1+\Brac{\norm[\Omega]{\bmf}^2 + \norm[\Omega]{\bmg}^2}^\frac12\right),
	\end{aligned}
	\end{equation}
	where $C$ is independent of $h$, $\bmf$ and $\bmg$, but depends on $\nu_k$ and $\nu_m$.
\end{theorem}

\begin{theorem}[Convergence of the HHO scheme]\label{thm:convergence}
Let, for each $h\in\mathcal{H}$, $(\ulbmuh, \ulbmbh, q_h, r_h) \in \bmX_h^k$ be a solution to the HHO scheme \eqref{eq:discrete.fluid}--\eqref{eq:discrete.magnetic.div}. There exists a solution $(\bmu, \bmb, q, r)\in\bmU\times\bmB\times P\times P$ to the continuous problem \eqref{eq:variational} such that, along a subsequence as $h\to 0$,
\begin{align*}
\bmuh\to \bmu\mbox{ and }\bmbh\to\bmb&\quad\mbox{ in $L^s(\Omega)^3$ for all $s\in[1,6)$},\\
\nabla_h \rh{k+1}\ulbmuh\to \nabla\bmu\mbox{ and }\nabla_h \rh{k+1}\ulbmbh\to\nabla\bmb&\quad\mbox{ in $L^2(\Omega)^{3\times 3}$},\\
q_h\to q\mbox{ and }r_h\to r&\quad\mbox{ in $L^2(\Omega)$},
\end{align*}
where, for $\ulbmzh\in\bmUhkzr$ or $\ulbmzh\in \bmUhkn$, $\rh{k+1}\ulbmzh:\Omega\to\R^3$ denotes the piecewise polynomial function defined by $(\rh{k+1}\ulbmzh)_{|T}=\rT{k+1}\ulbmzT$ for all $T\in\Th$, and $\nabla_h$ is the broken gradient.
\end{theorem}

\begin{theorem}[Uniqueness of the discrete solution]\label{thm:uniqueness}
	The discrete problem \eqref{eq:discrete.fluid}--\eqref{eq:discrete.magnetic.div} admits a unique solution $(\ulbmuh, \ulbmbh, q_h, r_h) \in \bmX_h^k$ provided the source terms $\bmf$ and $\bmg$ satisfy, for some $\chi\in[0, 1)$,
	\begin{equation}\label{eq:data.smallness}
		\Brac{\norm[\Omega]{\bmf}^2 + \norm[\Omega]{\bmg}^2}^\frac12 \le \chi\frac{\min(\nu_k, \nu_m)^2}{\sqrt{2}C_{\rma}^2C_{\rmt}C_p},
	\end{equation}
  where $C_{\rma}$, $C_{\rmt}$ and $C_p$ are the constants appearing in \eqref{eq:ah.stab.and.bound}, \eqref{eq:th.boundedness}, and  \eqref{eq:discrete.poincare} below.
\end{theorem}

\begin{theorem}[Energy error estimate for the HHO scheme]\label{thm:discrete.energy.error}
	Suppose that the source terms $\bmf$, $\bmg$ satisfy the data smallness condition \eqref{eq:data.smallness} with parameter $\chi\in[0,1)$. 
	Let $(\bmu, \bmb, q, r)\in\bmU\times\bmB\times P\times P$ solve the continuous problem \eqref{eq:variational}, and let $(\ulbmuh, \ulbmbh, q_h, r_h)\in\bmX_h^k$ be the unique solution to the discrete problem \eqref{eq:discrete}. Moreover, assume the additional regularity $\bmu,\bmb\in\WSP{1,4}(\Omega)^3\cap\HS{k+2}(\Th)^3$ and $q,r\in\HONE(\Omega)\cap\HS{k+1}(\Th)$. Then, setting
	\begin{align*}
	\mathcal N(\bmu,\bmb,q,r):={}&\nu_k\seminorm[\HS{k+2}(\Th)^3]{\bmu} + \nu_m\seminorm[\HS{k+2}(\Th)^3]{\bmb} + \norm[\WSP{1, 4}(\Omega)^3]{\bmu}\seminorm[\WSP{k+1, 4}(\Th)^3]{\bmu} + \norm[\WSP{1, 4}(\Omega)^3]{\bmb}\seminorm[\WSP{k+1, 4}(\Th)^3]{\bmb} \\
		& + \norm[\WSP{1, 4}(\Omega)^3]{\bmu}\seminorm[\WSP{k+1, 4}(\Th)^3]{\bmb} + \norm[\WSP{1, 4}(\Omega)^3]{\bmb}\seminorm[\WSP{k+1, 4}(\Th)^3]{\bmu} +  \seminorm[\HS{k+1}(\Th)]{q} + \seminorm[\HS{k+1}(\Th)]{r},
	\end{align*}
	the following error estimates hold:
	\begin{align}\label{eq:discrete.energy.error.ub}
		&(1-\chi)\norm[\rma,\rma,h]{(\ulbmuh - \bmIhk\bmu,\ulbmbh - \bmIhk\bmb)} \lesssim  h^{k+1}\calN(\bmu,\bmb,q,r),\\
		\label{eq:discrete.energy.error.qr}
		&(1-\chi)\Brac{\norm[\Omega]{q_h - \pihzr{k}q}+\norm[\Omega]{r_h - \pihzr{k}r}} \le C h^{k+1}\calN(\bmu,\bmb,q,r)\Brac{1+\Brac{\norm[\Omega]{\bmf}^2+\norm[\Omega]{\bmg}^2}^\frac12},
	\end{align}	
	where $C$ is independent of $h$, $\bmf$ and $\bmg$, but depends on $\nu_k$ and $\nu_m$, $\pihzr{k}$ is the $L^2$-projector on $\POLY{k}(\Th)$, and the energy-like norm $\norm[\rma,\rma,h]{(\cdot,\cdot)}:\bmUhkzr\times\bmUhkn\to\R$ is defined as	
	\begin{equation}\label{eq:aah.norm.def}
	\norm[\rma,\rma,h]{(\ulbmvh,\ulbmwh)}^2\defeq\nu_k\rma_h(\ulbmvh,\ulbmvh)+\nu_m\rma_h(\ulbmwh,\ulbmwh).
	\end{equation}
\end{theorem}

As a consequence of Theorem \ref{thm:discrete.energy.error}, we see that under the data-smallness condition and piecewise smoothness of the continuous solution, the entire sequence $(\bmuh, \bmbh, q_h, r_h)$ converges to $(\bmu, \bmb, q, r)$, and thus that the solution to the continuous problem is unique.

\subsection{An estimate on the fluid pressure}\label{sec:estimate.fluid.pressure}
The quantity $q$ approximated in \eqref{eq:discrete.energy.error.qr} incorporates the fluid pressure $p$ and the magnetic pressure $\frac12\rho \bmb \cdot \bmb$ via the equation
\[
	q = \frac{1}{\rho}\left(p + \frac12\rho\bmb \cdot \bmb\right).
\]
However, it can be interesting to define a discrete pressure $p_h$ and derive estimates on the quantity $\norm[\Omega]{p_h - \pihzr{k}p}$. Setting 
\begin{equation}\label{eq:discrete.pressure.def}
	p_h = \rho q_h - \frac12\rho\bmbh \cdot \bmbh,
\end{equation}
it holds by a triangle inequality,
\begin{equation}\label{eq:discrete.pressure.triangle.inequality}
	\norm[\Omega]{p_h - \pihzr{k}p} \le \rho\norm[\Omega]{q_h - \pihzr{k}q} + \frac12\rho\norm[\Omega]{\bmbh \cdot \bmbh - \pihzr{k}(\bmb \cdot \bmb)}.
\end{equation}
Consider,
\begin{align*}
	{}&\norm[\Omega]{\bmbh \cdot \bmbh - \pihzr{k}(\bmb \cdot \bmb)} \\
	 {}&\le \norm[\Omega]{\bmbh \cdot \bmbh - (\bmpihzr{k}\bmb) \cdot (\bmpihzr{k}\bmb)} + \norm[\Omega]{(\bmpihzr{k}\bmb) \cdot (\bmpihzr{k}\bmb) - \bmb \cdot \bmb} + \norm[\Omega]{\bmb \cdot \bmb - \pihzr{k}(\bmb \cdot \bmb)} 
	\\
	{}&\lesssim \norm[\LP{4}(\Omega)^3]{\bmbh - \bmpihzr{k}\bmb}\norm[\LP{4}(\Omega)^3]{\bmbh + \bmpihzr{k}\bmb} + \norm[\LP{4}(\Omega)^3]{\bmb - \bmpihzr{k}\bmb}\norm[\LP{4}(\Omega)^3]{\bmb + \bmpihzr{k}\bmb} + h^{k+1}\seminorm[\HS{k+1}(\Th)]{\bmb \cdot \bmb},
\end{align*}
where we have applied a H\"{o}lder inequality and the approximation properties of $\pihzr{k}$. Applying the approximation properties of $\bmpihzr{k}$ and equation \eqref{eq:discrete.poincare.sobolev} below (noting that $\bmbh - \bmpihzr{k}\bmb = (\ulbmbh - \bmIhk\bmb)_h$) as well as the $L^4(\Omega)$-boundedness of $\bmpihzr{k}$ and equation \eqref{eq:discrete.energy.error.ub} we infer that
\begin{multline*}
	\norm[\Omega]{\bmbh \cdot \bmbh - \pihzr{k}(\bmb \cdot \bmb)} \\ \le C h^{k+1} \BRAC{(1-\chi)^{-1}\calN(\bmu,\bmb,q,r)(\norm[\LP{4}(\Omega)^3]{\bmbh} + \norm[\LP{4}(\Omega)^3]{\bmb}) + \seminorm[\WSP{k+1,4}(\Th)^3]{\bmb}\norm[\LP{4}(\Omega)^3]{\bmb} + \norm[\WSP{k+1,4}(\Th)^3]{\bmb}^2}
\end{multline*}
where $C$ is independent of $h$, $\bmf$, $\bmg$ and $\chi$, but depends on $\nu_k$ and $\nu_m$. Combining with \eqref{eq:discrete.pressure.triangle.inequality} and Theorem \ref{thm:discrete.energy.error}, as well as equation \eqref{eq:discrete.apriori.bound} and Lemma \ref{lem:discrete.poincare.sobolev} below to bound the discrete term $\norm[\LP{4}(\Omega)^3]{\bmbh}$, yields an optimal $\mathcal O(h^{k+1})$ estimate on the discrete pressure.

\section{Analysis of the scheme}\label{sec:proofs}

\subsection{Discrete functional analysis results}

Let us first recall the following trace inequalities \cite[Lemma 4 and Eq.~(3.3)]{droniou.yemm:2021:robust}, valid for sets connected by star-shaped sets: for all \(p\in [1,\infty)\), it holds
\begin{alignat}{2}
\hT\norm[\LP{p}(\bdryT)]{v}^p \lesssim{}& \norm[\LP{p}(T)]{v}^p+\hT^p\norm[\LP{p}(T)^3]{\nabla v}^p&\qquad \forall v\in \WSP{1,p}(T),\label{eq:continuous.trace}\\
\hT\norm[\LP{p}(\bdryT)]{v}^p\lesssim{}& \norm[\LP{p}(T)]{v}^p&\qquad\forall v\in\POLY{k}(T).	\label{eq:discrete.trace}
\end{alignat}

\begin{lemma}[Sobolev--Poincar\'{e}--Wirtinger inequality for broken polynomial functions]\label{lem:sob.poincare.local.zero.avg}
	If $\wh\in\POLY{k}(\Th)$ satisfies $\int_T \wT = 0$ for all $T\in\Th$, then 
	\[
		\norm[\LP{6}(\Omega)]{\wh} \lesssim \norm[\Omega]{\nabla_h\wh}
	\]
	where $\nabla_h |_T = \nabla$ denotes the broken gradient, and the hidden constant depends on $\Omega, \varrho$, and $k$.
\end{lemma}

Lemma \ref{lem:sob.poincare.local.zero.avg} is a special case of \cite[Lemma 6.35]{di-pietro.droniou:2020:hybrid} (recalling that the space dimension is $d=3$ here), see also \cite[Lemma 5.8 and Remark 5.9]{di-pietro.droniou:2017:hybrid}. Inspecting the proof of \cite[Lemma 6.35]{di-pietro.droniou:2020:hybrid}, the hidden constant depends on a Poincar\'{e} inequality, an inverse Lebesgue inequality for polynomials, and the equivalence $|T|_d \approx \hT^d$, all of which hold under Assumption \ref{assum:star.shaped} \cite{droniou.yemm:2021:robust}.

We state here discrete Sobolev--Poincar\'e estimates on $\bmUhkn$. Since $\bmUhkzr\subset \bmUhkn$, these estimates also hold in $\bmUhkzr$.

\begin{lemma}[Discrete Sobolev--Poincar\'{e}--Wirtinger inequality]\label{lem:sob.poincare}
	For all $\ulbmvh\in\bmUhk$ it holds that
	\begin{equation}\label{eq:discrete.poincare.wirtinger}
		\norm[\LP{6}(\Omega)^3]{\bmvh - \olbm{v}_h} \lesssim \norm[1, h]{\ulbmvh}
	\end{equation}
	where
	\begin{equation}\label{eq:def.ovh}
		\olbm{v}_h = \frac{1}{|\Omega|}\int_{\Omega} \bmvh.
	\end{equation}
\end{lemma}

\begin{proof}
	A proof of \eqref{eq:discrete.poincare.wirtinger} is provided by \cite[Theorem 6.5]{di-pietro.droniou:2020:hybrid} for meshes possessing a regular matching simplicial submesh, by considering projections of the high-order HHO space on $\Th$ into the lowest-order space on the matching simplicial submesh. The arguments, however, extend easily to the case of star-shaped elements by instead projecting onto the lowest-order space on $\Th$ itself, by relying on the boundary trace inequality \eqref{eq:discrete.trace} rather than a trace inequality on each face, and by invoking \cite[Lemma B.24]{droniou.eymard:2018:gradient} (the latter lemma assumes the number of faces attached to each element $T\in\Th$ is bounded above; however, a closer inspection of the proof reveals this assumption is not necessary).
\end{proof}


\begin{lemma}\label{lem:discrete.poincare}
	It holds for all $\ulbmvh\in\bmUhkn$ that
	\begin{equation}\label{eq:bound.olbvh}
		\norm[\Omega]{\olbm{v}_h} \lesssim \norm[1, h]{\ulbmvh},
	\end{equation}
	where $\olbm{v}_h$ is defined in \eqref{eq:def.ovh}.
\end{lemma}

\begin{proof}
	Define $\psi(\boldsymbol{x})=\olbm{v}_h\cdot(\boldsymbol{x}-\boldsymbol{x}_\Omega)\in \POLY{1}(\Omega)$, where $\boldsymbol{x}_\Omega=\frac{1}{|\Omega|}\int_\Omega \boldsymbol{x}\,d\boldsymbol{x}$ is the centre of mass of $\Omega$. Then, $\int_\Omega \psi = 0$ and $\nabla \psi = \olbm{v}_h$,	and thus
	\begin{equation}\label{eq:bmUhkn.poinacre.proof.1}
		\norm[\Omega]{\olbm{v}_h}^2 = \int_\Omega\bmvh\cdot\olbm{v}_h = \sum_{T\in\Th}\int_T\bmvT \cdot \nabla \psi = \sum_{T\in\Th}\Brac{-\int_T (\nabla \cdot \bmvT) \psi + \int_{\bdryT} (\bmvT\cdot\norT) \psi}.
	\end{equation}
	Consider, by the homogeneous condition on $\bmUhkn$ and the single-valuedness of $\psi$ on the interior mesh faces,
	\[
		\sum_{T\in\Th}\int_{\bdryT}(\bmvT\cdot\norT) \psi  = \sum_{T\in\Th}\int_{\bdryT}((\bmvT-\bmvFT)\cdot\norT) \psi.
	\]
	By a Cauchy--Schwarz inequality and the continuous trace inequality \eqref{eq:continuous.trace},
	\begin{align}
		\sum_{T\in\Th}\int_{\bdryT}((\bmvT-\bmvFT)\cdot\norT) \psi	\lea \sum_{T\in\Th}\norm[\bdryT]{\bmvT-\bmvFT} \norm[\bdryT]{\psi} \nl
		\les \sum_{T\in\Th}\hT^{-\frac12}\norm[\bdryT]{\bmvT-\bmvFT} \Brac{\norm[T]{\psi} + \hT\norm[T]{\nabla\psi}} \nl
		\les \Brac{\sum_{T\in\Th}\hT^{-1}\norm[\bdryT]{\bmvT-\bmvFT}^2}^\frac12 \norm[\HONE(\Omega)]{\psi},\label{eq:bmUhkn.poinacre.proof.2}
	\end{align}
	where the final inequality results from a discrete Cauchy--Schwarz inequality and $\hT \lesssim 1$. We also have
	\begin{equation}\label{eq:bmUhkn.poinacre.proof.3}
		\sum_{T\in\Th}-\int_T (\nabla \cdot \bmvT) \psi \lesssim \sum_{T\in\Th}\norm[T]{\nabla \bmvT} \norm[T]{\psi} \le \Brac{\sum_{T\in\Th}\norm[T]{\nabla \bmvT}^2}^\frac12 \norm[\HONE(\Omega)]{\psi}.
	\end{equation}
	Substituting \eqref{eq:bmUhkn.poinacre.proof.2} and \eqref{eq:bmUhkn.poinacre.proof.3} into \eqref{eq:bmUhkn.poinacre.proof.1} yields
	\begin{equation}\label{eq:bmUhkn.poinacre.proof.4}
		\norm[\Omega]{\olbm{v}_h}^2 \lesssim \norm[1, h]{\ulbmvh}\norm[\HONE(\Omega)]{\psi}.
	\end{equation}
	As $\int_{\Omega}\psi = 0$, the following Poincar\'{e}--Wirtinger inequality holds: $\norm[\HONE(\Omega)]{\psi} \lesssim \seminorm[\HONE(\Omega)]{\psi} = \norm[\Omega]{\olbm{v}_h}$.
	Thus, we infer from \eqref{eq:bmUhkn.poinacre.proof.4} that $\norm[\Omega]{\olbm{v}_h}^2 \lesssim \norm[1, h]{\ulbmvh}\norm[\Omega]{\olbm{v}_h}$ and the proof of \eqref{eq:bound.olbvh} follows by simplifying by $\norm[\Omega]{\olbm{v}_h}$.
\end{proof}

\begin{lemma}\label{lem:discrete.poincare.sobolev}
	For all $\ulbmvh\in\bmUhkn$ it holds that
	\begin{equation}\label{eq:discrete.poincare.sobolev}
	\norm[\LP{6}(\Omega)^3]{\bmvh} \lesssim \norm[1, h]{\ulbmvh}.
	\end{equation}
\end{lemma}

\begin{proof}
	We use a triangle inequality to write $\norm[\LP{6}(\Omega)^3]{\bmvh} \le \norm[\LP{6}(\Omega)^3]{\bmvh - \olbm{v}_h} + \norm[\LP{6}(\Omega)^3]{\olbm{v}_h}$ and invoke then Lemma \ref{lem:sob.poincare} to bound the first addend, while for the second we use the fact that $\olbm{v}_h$ is constant together with \eqref{eq:bound.olbvh} to write $\norm[\LP{6}(\Omega)^3]{\olbm{v}_h} = |\Omega|^{\frac{1}{6} - \frac{1}{2}}\norm[\Omega]{\olbm{v}_h} \lesssim \norm[1,h]{\ulbm{v}_h}$.
\end{proof}

	We note that as a result of Lemma \ref{lem:discrete.poincare.sobolev}, $\bmUhkzr$ and $\bmUhkn$ are Banach spaces equipped with the norm $\norm[1, h]{\cdot}$. While this is a standard result on the space $\bmUhkzr$, to the best of our knowledge the result on $\bmUhkn$ is novel. In particular, we note the existence of a Poincar\'{e} constant $C_p>0$ independent of $h$ such that, for all $\ulbmvh\in\bmUhkn$,
	\begin{equation}\label{eq:discrete.poincare}
		\norm[\Omega]{\bmvh} \le C_p \norm[1, h]{\ulbmvh}.
	\end{equation}

We conclude our series of discrete functional analysis results on $\bmUhkn$ by establishing a discrete Rellich compactness result in that space, equivalent of \cite[Theorems 6.8 and 6.41]{di-pietro.droniou:2020:hybrid} in the case of Neumann or Dirichlet boundary conditions.

\begin{theorem}[Discrete Rellich theorem in $\bmUhkn$]\label{th:discrete.rellich}
Let $(\Mh)_{h\in\mathcal H}$ be a regular mesh sequence as per Assumption \ref{assum:star.shaped}	and, for all $h\in\mathcal H$, let $\ulbmvh\in\bmUhkn$. Assume
that $(\norm[1,h]{\ulbmvh})_{h\in\mathcal H}$ is bounded. Then, there exists $\bmv\in \bmB$ such that, up to a subsequence as $h\to 0$,
\begin{enumerate}
\item[(i)] $\bmvh\to \bmv$ and $\rh{k+1}\ulbmvh\to\bmv$ in $L^s(\Omega)^3$ for all $s\in[1,6)$,
\item[(ii)] $\nabla_h \rh{k+1}\ulbmvh\to \nabla\bmv$ and, for all integers $l\ge 0$, $\Gh{l}\ulbmvh\to\nabla\bmv$ weakly in $L^2(\Omega)^{3\times 3}$,
\end{enumerate}
where $\Gh{l}:\bmUhk\to\POLY{l}(\Th)^{3\times 3}$ is defined by $\Gh{l}\ulbmvh|_T=\GT{l}\ulbmvT$ for all $\ulbmvh\in\bmUhk$ and all $T\in\Th$.
\end{theorem}

\begin{proof}
The proof uses the corresponding compactness result \cite[Theorem 6.8]{di-pietro.droniou:2020:hybrid} for Neumann boundary conditions (as well as the straightforward adaptation of \cite[Theorem 9.29]{di-pietro.droniou:2020:hybrid}, stated for Dirichlet boundary conditions, to Neumann boundary conditions). These theorems are established under more restrictive mesh assumptions than Assumption \ref{assum:star.shaped} but, following the approach described in the proof of Lemma \ref{lem:sob.poincare} and using \cite[Lemma B.27]{droniou.eymard:2018:gradient}, they can easily be established under our notion of regular mesh sequence.

Define $\olbm{v}_h\in\R^3$ by \eqref{eq:def.ovh} and set $\ulbmwh=\ulbmvh-\bmIhk\olbm{v}_h$. We have $\norm[1,h]{\ulbmwh}=\norm[1,h]{\ulbmvh}$ and $\int_\Omega\bmwh=\int_\Omega (\bmvh-\olbm{v}_h)=\bm{0}$. Hence, applying \cite[Theorems 6.8 and 9.29]{di-pietro.droniou:2020:hybrid} with $p=2$ to each component of $\ulbmwh$, we find $\bmw\in H^1(\Omega)^3$ such that the convergences (i)--(ii) hold with all letters $\bmv$ replaced by $\bmw$, and such that $\bm{\gamma}_h\ulbmwh\to \bmw_{|\partial\Omega}$, where the discrete trace operator $\bm{\gamma}_h:\bmUhk\to L^2(\partial\Omega)^3$ is defined by: for all $\ulbmzh\in\bmUhk$, $(\bm{\gamma}_h\ulbmzh)_{|F}=\bmzF$ for all $F\in\Fhb$.

Since $(\olbm{v}_h)_h$ is also bounded by \eqref{eq:bound.olbvh}, upon extracting a subsequence from the previous one we can assume that $\olbm{v}_h\to \bm{\xi}\in\R^3$ as $h\to 0$ and thus, setting $\bmv=\bmw+\bm{\xi}\in H^1(\Omega)^3$, we obtain the convergences (i)--(ii), together with $\bm{\gamma}_h\ulbmvh\to \bmv_{|\partial\Omega}$. For all $h\in\mathcal H$ we have $\bm{\gamma}_h\ulbmvh \cdot\bmn=0$ on $\partial\Omega$ (since $\ulbmvh\in \bmUhkn$), which proves by passing to the limit that $\bmv\cdot\bmn=0$ on the boundary and thus that $\bmv\in\bmB$.
\end{proof}
	
\begin{lemma}[Boundedness of the interpolator]\label{lem:bound.IT}
It holds
\begin{equation*}
\norm[1,T]{\bmITk\bmv}\lesssim \norm[T]{\nabla\bmv}\qquad\forall \bmv\in H^1(T)^3.
\end{equation*}
\end{lemma}

\begin{proof}
The proof follows \cite[Proposition 2.2]{di-pietro.droniou:2020:hybrid}; we just provide some details to show that the relaxed mesh regularity assumptions and change of discrete norm, compared to this reference, do not impact the result. Using the boundedness $\norm[T]{\nabla\bmpiTzr{k}\bmv}\lesssim \norm[T]{\nabla\bmv}$ (see \cite[Eq.~(1.77)]{di-pietro.droniou:2020:hybrid}) we have
\[
\norm[1,T]{\bmITk\bmv}^2\lesssim \norm[T]{\nabla\bmv}^2+h_T^{-1}\norm[\partial T]{\bmpiFTzr{k}\bmv-\bmpiTzr{k}\bmv}^2.
\]
\cite[Lemma 7]{droniou.yemm:2021:robust} shows that $\norm[\partial T]{\bmpiFTzr{k}\bmv-\bmpiTzr{k}\bmv}\lesssim h_T^{\frac12}\norm[T]{\nabla\bmv}$, which concludes the proof.
\end{proof}

\subsection{Properties of the discrete forms}

We state here the design properties of each of the bilinear forms $\ah$ and $\dh$, and the trilinear form $\th$. Most of these properties have already been obtained under the assumption of zero Dirichlet boundary conditions in \cite[Chapters 8, 9]{di-pietro.droniou:2020:hybrid}. However, we require several of these results with boundary conditions that either only impose vanishing normal components, and/or vanishing tangential curl of vector fields (see \eqref{eq:mhd.bc.b}, \eqref{eq:mhd.bc.curlb}). Moreover, another difference with respect to \cite[Chapters 8, 9]{di-pietro.droniou:2020:hybrid} is that we consider here weaker assumptions on the meshes, and that the discrete norm defined by \eqref{eq:local.norm.def} differs from that of \cite[Eq.~(8.15)]{di-pietro.droniou:2020:hybrid} due to the change of scaling $\hF^{-1}\to\hT^{-1}$. We therefore provide proofs of these results, focusing on the modifications coming from the different boundary conditions and mesh assumptions. 

\begin{proposition}[Properties of $\ah$]\label{prop:ah.properties}
	The discrete bilinear form $\ah$ satisfies the following properties:
	\begin{enumerate}[label=\normalfont(A\arabic*),ref=\normalfont(A\arabic*)]
		\item\label{item:ah.property.stab} Stability and Boundedness. There exists a $C_{\rma}\ge 0$ independent of $h$ such that, for all $\ulbmvh \in \bmUhk$,
		\begin{equation}\label{eq:ah.stab.and.bound}
			C_{\rma}^{-1} \norm[1, h]{\ulbmvh}^2 \le \ah(\ulbmvh, \ulbmvh) \le C_{\rma} \norm[1, h]{\ulbmvh}^2.
		\end{equation}
		\item\label{item:ah.property.cons} Consistency. 
		For all $\bmw \in \bmB \cap H^{k+2}(\Th)^3$ such that $\Delta\bmw\in\LTWO(\Omega)^3$ and $\bmn_{\Omega}\times(\nabla \times \bmw)=\bm{0}$ on $\bdry\Omega$, and for all $\ulbmvh \in \bmUhkn$,
		\begin{equation}\label{eq:ah.consistency}
			\left|-\int_{\Omega} \Delta \bmw \cdot \bmvh - \ah(\bmIhk \bmw, \ulbmvh)\right| \lesssim \norm[1,h]{\ulbmvh} h^{k+1} \seminorm[H^{k+2}(\Th)^3]{\bmw}. 
		\end{equation}	
	\end{enumerate}
\end{proposition}

\begin{proof}	
	The stability and boundedness \ref{item:ah.property.stab} follows trivially from the properties of $\aT$ stated in Assumption \eqref{eq:aT.norm.equivalence}. Let us turn to \ref{item:ah.property.cons}.
	An integration by parts in each element $T\in\Th$ yields
	\[
			-\int_{\Omega} \Delta \bmw \cdot \bmvh 
			= \sum_{T\in\Th}\SqBrac{\int_{T} \nabla \bmvT : \nabla \bmw - \int_{\bdryT}\bmvT\cdot(\norT \cdot \nabla) \bmw}.
	\]
	Let us also consider	
	\begin{align*}
		 \sum_{T\in\Th}\int_{\bdryT}\bmvFT\cdot(\norT \cdot \nabla) \bmw \eq \sum_{T\in\Th}\sum_{F\in\Fh[T]} \int_{F}\bmvF\cdot(\norTF \cdot \nabla) \bmw \nl
		 \eq \sum_{F\in\Fh}\sum_{T\in\Th[F]} \int_{F}\bmvF\cdot(\norTF \cdot \nabla) \bmw \nl
		 \eq \sum_{F\in\Fhb}\int_{F}\bmvF\cdot(\norF \cdot \nabla) \bmw = 0,
	\end{align*}
	where the integrals on each of the internal faces have cancelled in the third line due to continuity of $(\nabla\bmw) \norTF$ (since $\Delta\bmw\in\LTWO(\Omega)^3$ and $\bmw\in\HS{2}(\Th)^3$), and the conclusion follows from the boundary conditions on $\ulbmvh$ and $\bmw$, and equation \eqref{eq:bdry.cancellation.polyhedral}. Therefore, we may write
	\begin{equation}\label{eq:ah.cons.proof.1}
		-\int_{\Omega} \Delta \bmw \cdot \bmvh = \sum_{T\in\Th}\SqBrac{\int_{T} \nabla \bmvT : \nabla \bmw + \int_{\bdryT}(\bmvFT - \bmvT)\cdot(\bmn_T \cdot \nabla) \bmw}.
	\end{equation}
	
	By the commutation property \eqref{eq:rT.and.DT.commutation}, and invoking the definition \eqref{eq:rT.def} of $\rT{k+1}\ulbmvT$, it holds that	
	\begin{equation}
	\begin{aligned}
	\label{eq:ah.cons.proof.2}
		\ah(\bmIhk \bmw, \ulbmvh)={}& \sum_{T\in\Th}\SqBrac{\int_{T} \nabla \bmvT : \nabla \bmpiTe{k+1}\bmw + \int_{\bdryT}(\bmvFT - \bmvT)\cdot(\bmn_T \cdot \nabla) \bmpiTe{k+1}\bmw}\\
		&+\sum_{T\in\Th}\rms_T(\bmITk \bmw, \ulbmvT).
  \end{aligned}
	\end{equation}	
	Thus, subtracting \eqref{eq:ah.cons.proof.2} from \eqref{eq:ah.cons.proof.1}, noting the $H^1$-orthogonality of $\bmpiTe{k+1}$, 	
	\begin{align*}
		-\int_{\Omega} \Delta \bmw{}& \cdot \bmvh - \ah(\bmIhk \bmw, \ulbmvh) = \sum_{T\in\Th}\int_{\bdryT}(\bmvFT - \bmvT)\cdot(\bmn_T \cdot \nabla) (\bmw - \bmpiTe{k+1}\bmw) -\sum_{T\in\Th}\rms_T(\bmITk \bmw, \ulbmvT)\nl
		\les \sum_{T\in\Th}\hT^{-\frac12}\norm[\bdryT]{\bmvFT - \bmvT}\Brac{\seminorm[\HS{1}(T)]{\bmw - \bmpiTe{k+1}\bmw} + \hT\seminorm[\HS{2}(T)]{\bmw - \bmpiTe{k+1}\bmw}}\nl
		&+\left(\sum_{T\in\Th}\rms_T(\bmITk \bmw, \bmITk\bmw)\right)^\frac12
		\left(\sum_{T\in\Th}\rms_T(\ulbmvT, \ulbmvT)\right)^\frac12,
	\end{align*}
	where we have invoked a Cauchy--Schwarz inequality and the continuous trace inequality \eqref{eq:continuous.trace}. The proof then follows from a discrete Cauchy--Schwarz inequality, the volumetric approximation properties \cite[Theorem 1.48]{di-pietro.droniou:2020:hybrid} of the elliptic projector, the consistency property \eqref{eq:sT.consistency} of the stabilisation form, and the norm equivalence \eqref{eq:aT.norm.equivalence}.
\end{proof}

\begin{proposition}[Properties of $\dh$]\label{prop:dh.properties}
The discrete bilinear form $\dh$ satisfies the following properties:
\begin{enumerate}[label=\normalfont(D\arabic*),ref=\normalfont(D\arabic*)]
	\item\label{item:dh.property.stab} Inf-sup stability. There exists a $C_{\rmd}> 0$ independent of $h$ such that, for all $q_h \in \POLY{k}_0(\Th)$,
	\begin{equation}\label{eq:dh.boundedness}
		C_{\rmd}^{-1}\norm[\Omega]{q_h} \le \sup_{\ulbmvh \in \bmUhkzr \backslash \{\ulbm{0}\}} \frac{\dh(\ulbmvh, q_h)}{\norm[1,h]{\ulbmvh}}\le \sup_{\ulbmvh \in \bmUhkn \backslash \{\ulbm{0}\}} \frac{\dh(\ulbmvh, q_h)}{\norm[1,h]{\ulbmvh}}.
	\end{equation}
	\item\label{item:dh.property.cons} Consistency. For all $s\in\HONE(\Omega)\cap\HS{k+1}(\Th)$ and $\ulbmvh\in\bmUhkn$,
	\begin{equation}\label{eq:dh.consistency}
		\left|\int_\Omega \bmvh \cdot \nabla s - \dh(\ulbmvh, \pihzr{k}s)\right| \\ \lesssim h^{k+1} \seminorm[H^{k+1}(\Th)]{s}\norm[1,h]{\ulbmvh}. 
	\end{equation}			
\end{enumerate}
\end{proposition}

\begin{proof}	
	The first inequality in \ref{item:dh.property.stab} is established in \cite[Lemma 8.12]{di-pietro.droniou:2020:hybrid}, for a slightly different discrete $H^1$-norm with the scaling $h_T^{-1}$ in \eqref{eq:local.norm.def} replaced by local face-based scaling $h_F^{-1}$. This change actually only impacts, in this proof, the boundedness of the interpolator $\bmITk$, which we established in Lemma \ref{lem:bound.IT} for the norm with the scaling $h_T^{-1}$.
	The second inequality in \ref{item:dh.property.stab} follows from the inclusion $\bmUhkzr\subset \bmUhkn$.

  We now turn to \ref{item:dh.property.cons}.	
%
%
	It follows from the definition \eqref{eq:ah.and.dh.def} of $\dh$ and integrating by parts the defining equation \eqref{eq:DT.def} of $\DT{k}$ that
	\begin{equation}\label{eq:dh.property.cons.proof.1}
		\dh(\ulbmvh, \pihzr{k}s) = \sum_{T\in\Th}-\int_T \DT{k}\ulbmvT \piTzr{k}s 
		= \sum_{T\in\Th}\SqBrac{-\int_T (\nabla \cdot\bmvT) \piTzr{k}s + \int_{\partial T}(\bmvT - \bmvFT)\cdot\norT \piTzr{k}s}.
	\end{equation}
	
	Since $\ulbmvh\in\bmUhkn$ and $s\in\HONE(\Omega)$ we have 
	\[
			\sum_{T\in\Th}\int_{\bdryT}(\bmvFT\cdot\norT) s = 0.
	\]
	Hence, integrating by parts on each $T\in\Th$ and introducing the term above,
	\begin{equation}\label{eq:dh.property.cons.proof.2}
	\int_\Omega \bmvh \cdot \nabla s=\sum_{T\in\Th}\SqBrac{-\int_T (\nabla \cdot\bmvT) s + \int_{\bdryT}(\bmvT - \bmvFT)\cdot\norT s}.
	\end{equation}
	Therefore, combining equations \eqref{eq:dh.property.cons.proof.1} and \eqref{eq:dh.property.cons.proof.2}, and noting that $s-\piTzr{k}s$ is $\LTWO(T)$-orthogonal to $\nabla\cdot\bmvT\in\POLY{k}(T)$,
	\begin{align*}
		\Big|\int_\Omega \bmvh \cdot \nabla s - \dh(\ulbmvh, \pihzr{k}s)\Big| \eq \Big|\sum_{T\in\Th} \int_{\partial T}(\bmvT - \bmvFT)\cdot\norT (s - \piTzr{k}s) \Big|\nl
		 \les \sum_{T\in\Th}\hT^{-\frac12}\norm[\bdryT]{\bmvFT - \bmvT} \Brac{\norm[\T]{s - \piTzr{k}s} + \hT\norm[\T]{\nabla(s - \piTzr{k}s)}}
	\end{align*}
	where we have applied a Cauchy--Schwarz inequality to the integral and invoked the continuous trace inequality \eqref{eq:continuous.trace}. The proof then follows by invoking the volumetric approximation properties \cite[Theorem 1.45]{di-pietro.droniou:2020:hybrid} of $\piTzr{k}$ and applying a discrete Cauchy--Schwarz inequality to the sum.
\end{proof}

\begin{proposition}[Properties of $\th$]\label{prop:th.properties}
	The discrete trilinear form $\th$ satisfies the following properties:
	\begin{enumerate}[label=\normalfont(T\arabic*),ref=\normalfont(T\arabic*)]
		\item\label{item:th.property.skew.sym} Skew-symmetry. For all $\ulbmvh, \ulbmwh, \ulbmzh \in \bmUhk$
		\begin{equation}\label{eq:th.skew.sym}
			\th(\ulbmvh, \ulbmwh, \ulbmzh) = -\th(\ulbmvh, \ulbmzh, \ulbmwh).
		\end{equation}
		\item\label{item:th.property.bound} Boundedness. There exists $C_{\rmt}\ge 0$ independent of $h$ such that, for all $\ulbmvh, \ulbmwh, \ulbmzh \in \bmUhkn$,
		\begin{equation}\label{eq:th.boundedness}
			|\th(\ulbmvh, \ulbmwh, \ulbmzh)| \le C_{\rmt} \norm[1, h]{\ulbmvh} \norm[1, h]{\ulbmwh} \norm[1, h]{\ulbmzh}.
		\end{equation}
		\item\label{item:th.property.cons} Consistency. For all $\bmv, \bmw \in W^{1,4}(\Omega)^3 \cap W^{k+1,4}(\Th)^3$ such that $\nabla \cdot \bmv= 0$ and $\bmv \cdot \nor_{\Omega} = 0$ and for all $\ulbmzh \in \bmUhk$,
		\begin{multline}\label{eq:th.consistency}
		 \SEMINORM{\int_{\Omega}(\bmv\cdot\nabla)\bmw\cdot\bmzh - \th(\bmIhk\bmv, \bmIhk\bmw, \ulbmzh)} \\ \lesssim \norm[1,h]{\ulbmzh} h^{k+1} \Brac{\norm[\WSP{1, 4}(\Omega)^3]{\bmv} \seminorm[\WSP{k + 1, 4}(\Th)^3]{\bmw} + \norm[\WSP{1, 4}(\Omega)^3]{\bmw} \seminorm[\WSP{k + 1, 4}(\Th)^3]{\bmv}}. 
		\end{multline}			
	\end{enumerate}
\end{proposition}

\begin{proof}	
	The skew-symmetry \ref{item:th.property.skew.sym} is trivial from	the definitions \eqref{eq:th.def} of $\th$ and \eqref{eq:tT.def} of $\tT$.
	The proof of \ref{item:th.property.bound} is done as in \cite[Proposition 9.17]{di-pietro.droniou:2020:hybrid} (see also Lemma 9.15 therein), using generalised H\"older inequalities with exponents $(4,2,4)$, the boundedness \cite[Eq.~(9.37)]{di-pietro.droniou:2020:hybrid} of $\GT{2k}$ (which is easily checked to hold with the modified discrete norm \eqref{eq:local.norm.def}), and the Poincar\'e--Sobolev inequality \eqref{eq:discrete.poincare.sobolev}.

We now turn to \ref{item:th.property.cons}.	
	It follows the integration by parts formula \eqref{eq:convective.ibp} with $\Omega$ replaced by a generic $T$ that
	\begin{equation*}
		\int_{\Omega}(\bmv\cdot\nabla)\bmw\cdot\bmzh = \frac12 \sum_{T\in\Th}\SqBrac{\int_T (\bmv \cdot \nabla) \bmw \cdot \bmzT - \int_T (\bmv \cdot \nabla) \bmzT \cdot \bmw + \int_{\bdryT}(\bmv\cdot\norT)(\bmzT - \bmzFT)\cdot\bmw},
	\end{equation*}
	where we justify including the term
	\[
		\sum_{T\in\Th}\int_{\bdryT}(\bmv\cdot\norT)\bmzFT\cdot\bmw = \sum_{F\in\Fh}\sum_{T\in\Th[F]}\int_{F}(\bmv\cdot\norTF)\bmzF\cdot\bmw = 0
	\]
	by the single valuedness of $\bmv$ and $\bmw$ on each interior face, and $\bmv\cdot\norTF=0$ for all $F\subset\partial\Omega$. The discrete trilinear form $\th$ may be expanded using equation \eqref{eq:GT.expanded} as
	\begin{align*}
		\th(\bmIhk\bmv, \bmIhk\bmw, \ulbmzh) \eq \frac12 \sum_{T\in\Th}\Big[\int_T\Brac{(\bmpiTzr{k}\bmv\cdot \GT{2k})\bmITk \bmw \cdot \bmzT - (\bmpiTzr{k}\bmv \cdot \nabla)\bmzT\cdot\bmpiTzr{k}\bmw} \nl
		\minus \int_{\bdryT}(\bmpiTzr{k}\bmv \cdot \norT)(\bmzFT - \bmzT)\cdot\bmpiTzr{k}\bmw\Big].
	\end{align*}
	Therefore
	\begin{align*}
		\int_{\Omega}&(\bmv\cdot\nabla)\bmw\cdot\bmzh - \th(\bmIhk\bmv, \bmIhk\bmw, \ulbmzh) \nl
		\eq \frac12 \sum_{T\in\Th}\int_T \SqBrac{(\bmv \cdot \nabla)\bmw\cdot \bmzT - (\bmpiTzr{k}\bmv\cdot \GT{2k})\bmITk \bmw \cdot \bmzT} \nl
		\minus \frac12 \sum_{T\in\Th}\int_T \SqBrac{((\bmv- \bmpiTzr{k}\bmv) \cdot \nabla) \bmzT \cdot \bmw + (\bmpiTzr{k}\bmv \cdot \nabla)\bmzT\cdot(\bmw-\bmpiTzr{k}\bmw)} \nl
		\minus \frac12 \sum_{T\in\Th}\int_{\bdryT}\SqBrac{((\bmv- \bmpiTzr{k}\bmv)\cdot\norT) (\bmzFT - \bmzT) \cdot \bmw + (\bmpiTzr{k}\bmv\cdot\norT) (\bmzFT - \bmzT) \cdot (\bmw- \bmpiTzr{k}\bmw)}.
	\end{align*}
	The conclusion follows as in \cite[Proposition 9.17]{di-pietro.droniou:2020:hybrid}, using the consistency of $\GT{2k}$ (Lemma 9.15 in this reference), the Poincar\'e--Sobolev inequality \eqref{eq:discrete.poincare.sobolev}, the continuous trace inequality \eqref{eq:continuous.trace}, and the approximation properties of the $L^2$-projectors \cite[Theorem 1.45]{di-pietro.droniou:2020:hybrid}.
	\end{proof}

\subsection{Consistency errors}

For all $(\bmv, \bmw, p, s)\in\bmU\times\bmB\times P\times P$, the kinetic consistency error $\calE_{k, h}((\bmv, \bmw, p, s); \cdot):\bmUhkzr\to\R$ and magnetic consistency error $\calE_{m, h}((\bmv, \bmw, p, s); \cdot):\bmUhkn\to\R$ are defined via
\begin{equation}\label{eq:kinetic.cons.def}
	\calE_{k, h}((\bmv, \bmw, p, s); \ulbmzh) = \nu_k\calE_{\rma,h}(\bmv;\ulbmzh) + \calE_{\rmt,h}((\bmv,\bmv);\ulbmzh) - \calE_{\rmt,h}((\bmw,\bmw);\ulbmzh) + \calE_{\rmd,h}(p;\ulbmzh)
\end{equation}
and
\begin{equation}\label{eq:magnetic.cons.def}
\calE_{m, h}((\bmv, \bmw, p, s); \ulbmzh) = \nu_m\calE_{\rma,h}(\bmw;\ulbmzh) + \calE_{\rmt,h}((\bmv,\bmw);\ulbmzh) - \calE_{\rmt,h}((\bmw,\bmv);\ulbmzh) + \calE_{\rmd,h}(s;\ulbmzh)
\end{equation}
where the linear forms $\calE_{\rma,h}(\bmv;\cdot):\bmUhk\to\R$, $\calE_{\rmd,h}(s;\cdot):\bmUhk\to\R$ and $\calE_{\rmt,h}((\bmv,\bmw);\cdot):\bmUhk\to\R$ are defined as
\[
	\calE_{\rma,h}(\bmv;\ulbmzh) \defeq -\int_{\Omega} \Delta \bmv \cdot \bmzh - \ah(\bmIhk \bmv, \ulbmzh),
\]

\[
	\calE_{\rmd,h}(s;\ulbmzh) \defeq \int_\Omega \bmzh \cdot \nabla s - \dh(\ulbmzh, \pihzr{k}s),
\]
and
\[
	\calE_{\rmt,h}((\bmv,\bmw);\ulbmzh) \defeq \int_\Omega (\bmv\cdot\nabla)\bmw\cdot\bmzh - \th(\bmIhk \bmv, \bmIhk \bmw, \ulbmzh).
\]

\begin{theorem}[Consistency error]
	Suppose that $(\bmv, \bmw, p, s)\in\bmU\times\bmB\times P\times P$ satisfy the additional regularity: $\bmv, \bmw \in \HS{k+2}(\Th)^3\cap\WSP{1,4}(\Omega)^3$, $\Delta\bmv, \Delta\bmw \in \LTWO(\Omega)^3$, $\nabla\cdot\bmv=\nabla\cdot\bmw=0$, $\bmn_{\Omega}\times(\nabla\times\bmw)=\bm{0}$ on $\bdry\Omega$, and $p, s\in\HONE(\Omega)\cap\HS{k+1}(\Th)$. Then the kinetic consistency error $\calE_{k, h}((\bmv, \bmw, p, s); \cdot)$ and magnetic consistency error $\calE_{m, h}((\bmv, \bmw, p, s); \cdot)$ satisfy the following estimates:
	\begin{multline}\label{eq:kinetic.cons.error}
		\sup_{\ulbmzh \in \bmUhkzr \backslash \{\ulbm{0}\}}\frac{|\calE_{k, h}((\bmv, \bmw, p, s); \ulbmzh)|}{\norm[1, h]{\ulbmzh}}
		\lesssim h^{k+1}\Big[\nu_k\seminorm[\HS{k+2}(\Th)^3]{\bmv} + \norm[\WSP{1, 4}(\Omega)^3]{\bmv}\seminorm[\WSP{k+1, 4}(\Th)^3]{\bmv} \\ + \norm[\WSP{1, 4}(\Omega)^3]{\bmw}\seminorm[\WSP{k+1, 4}(\Th)^3]{\bmw} + \seminorm[\HS{k+1}(\Th)]{p}\Big],
	\end{multline}
	and
	\begin{multline}\label{eq:magnetic.cons.error}
		\sup_{\ulbmzh \in \bmUhkn \backslash \{\ulbm{0}\}}\frac{|\calE_{m, h}((\bmv, \bmw, p, s); \ulbmzh)|}{\norm[1, h]{\ulbmzh}}
		\lesssim h^{k+1}\Big[\nu_m\seminorm[\HS{k+2}(\Th)^3]{\bmw} + \norm[\WSP{1, 4}(\Omega)^3]{\bmv}\seminorm[\WSP{k+1, 4}(\Th)^3]{\bmw} \\ + \norm[\WSP{1, 4}(\Omega)^3]{\bmw}\seminorm[\WSP{k+1, 4}(\Th)^3]{\bmv} + \seminorm[\HS{k+1}(\Th)]{s}\Big].
	\end{multline}
\end{theorem}

\begin{proof}
	The proof follows trivially by a triangle inequality on the terms $|\calE_{k, h}((\bmv, \bmw, p, s); \ulbmzh)|$ and $|\calE_{m, h}((\bmv, \bmw, p, s); \ulbmzh)|$, and applying the estimates \eqref{eq:ah.consistency} (and \cite[Eq.~(8.32)]{di-pietro.droniou:2020:hybrid} for the Dirichlet boundary conditions on $\bmv$), \eqref{eq:dh.consistency} and \eqref{eq:th.consistency}.
\end{proof}

\subsection{Stability of the scheme}

The existence of a solution to the scheme and its uniqueness and error estimates for small data all follow from a general stability result that we now establish. To simplify the presentation and proof of this result, we start by re-casting the variational formulation \eqref{eq:discrete} in a more compact way: find $\underline{x}_h\in\bmX_h^k$ such that
\begin{equation}\label{eq:var.glob}
\rmA_h(\underline{x}_h,\underline{y}_h)+\rmT_h(\underline{x}_h,\underline{x}_h,\underline{y}_h)=\rmF_h(\underline{y}_h)\quad\forall\underline{y}_h\in\bmX_h^k,
\end{equation}
where we have defined the linear form $\rmF_h:\bmX_h^k\to\R$, bilinear form $\rmA_h:\bmX_h^k\times\bmX_h^k\to \R$ and trilinear form $\rmT_h:\bmX_h^k\times\bmX_h^k\times \bmX_h^k\to \R$ by,
for $\underline{x}^i_h=(\ulbmuh^i, \ulbmbh^i, q^i_h, r^i_h)\in\bmX_h^k$ (with $i=\sharp,\flat$) and $\underline{y}_h=(\ulbmvh,\ulbmwh,s_h,z_h)\in\bmX_h^k$,
\[
\rmF_h(\underline{y}_h)=(\bmf,\bmv_h)_\Omega+(\bmg,\bmw_h)_\Omega,
\]
\[
\rmA_h(\underline{x}^\sharp_h,\underline{y}_h)=\nu_k\rma_h(\ulbmuh^\sharp,\ulbmvh)+\nu_m\rma_h(\ulbmbh^\sharp,\ulbmwh)+\rmd_h(\ulbmvh,q_h^\sharp)
+\rmd_h(\ulbmwh,r_h^\sharp)-\rmd_h(\ulbmuh^\sharp,s_h)-\rmd_h(\ulbmbh^\sharp,z_h)
\]
and
\[
\rmT_h(\underline{x}^\sharp_h,\underline{x}^\flat_h,\underline{y}_h)=\rmt_h(\ulbmuh^\sharp,\ulbmuh^\flat,\ulbmvh)-\rmt_h(\ulbmbh^\sharp,\ulbmbh^\flat,\ulbmvh)
+\rmt_h(\ulbmuh^\sharp,\ulbmbh^\flat,\ulbmwh)-\rmt_h(\ulbmbh^\sharp,\ulbmuh^\flat,\ulbmwh).
\]

Simple algebra shows that
\begin{equation}\label{eq:coer.rmAh}
\rmA_h(\underline{x}^\sharp_h,\underline{x}^\sharp_h)=\nu_k\rma_h(\ulbmuh^\sharp,\ulbmuh^\sharp)+\nu_m\rma_h(\ulbmbh^\sharp,\ulbmbh^\sharp)=\norm[\rma,\rma,h]{(\ulbmuh^\sharp,\ulbmbh^\sharp)}^2,
\end{equation}
where $\norm[\a,\a,h]{(\cdot,\cdot)}$ is defined by \eqref{eq:aah.norm.def} and, using the skew-symmetry \eqref{eq:th.skew.sym} of $\rmt_h$, that
\begin{equation}\label{eq:skew.sym.rmTh}
\rmT_h(\underline{x}_h,\underline{y}_h,\underline{y}_h)=0.
\end{equation}

We define the norm $\norm[1,1,h]{{\cdot}}$ on $\bmUhkzr\times\bmUhkn$ by
\[
\norm[1,1,h]{(\ulbmvh,\ulbmwh)}=(\norm[1,h]{\ulbmvh}^2+\norm[1,h]{\ulbmwh}^2)^{\frac12}\qquad\forall (\ulbmvh,\ulbmwh)\in \bmUhkzr\times\bmUhkn
\]
and, for a linear mapping $\rmG:\bmX_h^k\to\R$ depending only on the first two components (such as $\rmF_h$ above), the dual norm by
\[
\norm[\bmX,\star]{\rmG}\defeq \sup\left\{\frac{\rmG((\ulbmvh,\ulbmwh,0,0))}{\norm[1,1,h]{(\ulbmvh,\ulbmwh)}}\,:\,
(\ulbmvh,\ulbmwh)\in\bmUhkzr\times\bmUhkn\backslash\{(\ul{0},\ul{0})\}
\right\}.
\]

\begin{lemma}[Boundedness of $\rmT_h$]
It holds, for all $\underline{x}_h^\sharp,\underline{x}_h^\flat,\underline{y}_h\in\bmX_h^k$,
\begin{equation}\label{eq:bound.rmTh}
|\rmT_h(\underline{x}_h^\sharp,\underline{x}_h^\flat,\underline{y}_h)|\le \sqrt{2}C_{\rmt}\norm[1,1,h]{(\ulbmuh^\sharp,\ulbmbh^\sharp)}\norm[1,1,h]{(\ulbmuh^\flat,\ulbmbh^\flat)}\norm[1,1,h]{(\ulbmvh,\ulbmwh)}.
\end{equation}
\end{lemma}

\begin{proof}
Using \eqref{eq:th.boundedness} we have
\begin{align}
C_{\rmt}^{-1}|\rmT_h(\underline{x}_h^\sharp,\underline{x}_h^\flat,\underline{y}_h)|\le{}&
 \norm[1,h]{\ulbmuh^\sharp}\norm[1,h]{\ulbmuh^\flat}\norm[1,h]{\ulbmvh}
+ \norm[1,h]{\ulbmbh^\sharp}\norm[1,h]{\ulbmbh^\flat}\norm[1,h]{\ulbmvh}\nonumber\\
&+ \norm[1,h]{\ulbmuh^\sharp}\norm[1,h]{\ulbmbh^\flat}\norm[1,h]{\ulbmwh}
+ \norm[1,h]{\ulbmbh^\sharp}\norm[1,h]{\ulbmuh^\flat}\norm[1,h]{\ulbmwh}.
\label{eq:bound.rmTh.1}
\end{align}
We then use Cauchy--Schwarz inequalities to write
\begin{align*}
\norm[1,h]{\ulbmuh^\sharp}\norm[1,h]{\ulbmuh^\flat}+\norm[1,h]{\ulbmbh^\sharp}\norm[1,h]{\ulbmbh^\flat}\le{}&\norm[1,1,h]{(\ulbmuh^\sharp,\ulbmbh^\sharp)}\norm[1,1,h]{(\ulbmuh^\flat,\ulbmbh^\flat)},\\
\norm[1,h]{\ulbmuh^\sharp}\norm[1,h]{\ulbmbh^\flat}
+ \norm[1,h]{\ulbmbh^\sharp}\norm[1,h]{\ulbmuh^\flat}\le{}&
\norm[1,1,h]{(\ulbmuh^\sharp,\ulbmbh^\sharp)}\norm[1,1,h]{(\ulbmuh^\flat,\ulbmbh^\flat)}.
\end{align*}
Plugged into \eqref{eq:bound.rmTh.1}, this gives
\[
C_{\rmt}^{-1}|\rmT_h(\underline{x}_h^\sharp,\underline{x}_h^\flat,\underline{y}_h)|\le
\norm[1,1,h]{(\ulbmuh^\sharp,\ulbmbh^\sharp)}\norm[1,1,h]{(\ulbmuh^\flat,\ulbmbh^\flat)}\left(\norm[1,h]{\ulbmvh}+\norm[1,h]{\ulbmwh}\right)
\]
and the Cauchy--Schwarz inequality $\norm[1,h]{\ulbmvh}+\norm[1,h]{\ulbmwh}\le\sqrt{2}\norm[1,1,h]{(\ulbmvh,\ulbmwh)}$ concludes the proof.
\end{proof}

Both the uniqueness of the discrete solution and the error estimate will be consequences of the following stability result. 

\begin{lemma}[Stability of the scheme]\label{lem:stab.variational}
Let $\underline{x}_h^\sharp$ and $\underline{x}_h^\flat$ be two solutions of \eqref{eq:var.glob} corresponding to two right-hand sides $\rmF_h^\sharp$ and $\rmF_h^\flat$. Assume that, for some $\chi\in [0,1)$,
\begin{equation}\label{eq:var.glob.small1}
\frac{\sqrt{2}C_{\rmt}C_{\rma}}{\min(\nu_k,\nu_m)}\norm[1,1,h]{(\ulbmuh^\sharp,\ulbmbh^\sharp)}\le \chi.
\end{equation}
Then,
\begin{equation}\label{eq:var.glob.est1}
(1-\chi)\norm[\rma,\rma,h]{(\ulbmuh^\sharp-\ulbmuh^\flat,\ulbmbh^\sharp-\ulbmbh^\flat)}\le \frac{C_{\rma}^\frac12}{\min(\nu_k,\nu_m)^\frac12}\norm[\bmX,\star]{\rmF_h^\sharp-\rmF_h^\flat},
\end{equation}
and, for some $C_0$ depending only on $\nu_k$, $\nu_m$, $C_{\rmt}$, $C_{\rma}$ and $C_{\rmd}$,
\begin{multline}
\label{eq:var.glob.est.pressure}
(1-\chi)\Big(\norm[L^2(\Omega)]{q_h^\sharp-q_h^\flat}+\norm[L^2(\Omega)]{r_h^\sharp-r_h^\flat}\Big)\\
\le C_0\norm[\bmX,\star]{\rmF_h^\sharp-\rmF_h^\flat}\left(1+\norm[1,1,h]{(\ulbmuh^\sharp,\ulbmbh^\sharp)}+\norm[1,1,h]{(\ulbmuh^\flat,\ulbmbh^\flat)}\right).
\end{multline}
\end{lemma}

\begin{proof}
The vector $\underline{\omega}_h=\underline{x}_h^\sharp-\underline{x}_h^\flat=(\ul{\bme}_u,\ul{\bme}_b,e_q,e_r)\in\bmX_h^k$ satisfies the following error equation: for all $\underline{y}_h\in\bmX_h^k$,
\begin{equation}\label{eq:error.equation}
\rmA_h(\underline{\omega}_h,\underline{y}_h)+\rmT_h(\underline{\omega}_h,\underline{x}_h^\sharp,\underline{y}_h)+\rmT_h(\underline{x}^\flat_h,\underline{\omega}_h,\underline{y}_h)=(\rmF_h^\sharp-\rmF_h^\flat)(\underline{y}_h).
\end{equation}
Making $\underline{y}_h=\underline{\omega}_h$ and using \eqref{eq:coer.rmAh}, \eqref{eq:skew.sym.rmTh}, \eqref{eq:bound.rmTh} and the definition of $\norm[\bmX,\star]{{\cdot}}$, we infer
\[
\norm[\rma,\rma,h]{(\ul{\bme}_u,\ul{\bme}_b)}^2\le \norm[\bmX,\star]{\rmF_h^\sharp-\rmF_h^\flat}\norm[1,1,h]{(\ul{\bme}_u,\ul{\bme}_b)}+
\sqrt{2}C_{\rmt}\norm[1,1,h]{(\ulbmuh^\sharp,\ulbmbh^\sharp)}\norm[1,1,h]{(\ul{\bme}_u,\ul{\bme}_b)}^2.
\]
The estimate \eqref{eq:var.glob.est1} follows from \eqref{eq:var.glob.small1} by using \eqref{eq:ah.stab.and.bound} to write 
\begin{equation}\label{eq:bound.11.aa}
\norm[1,1,h]{{\cdot}}^2\le \frac{C_{\rma}}{\min(\nu_k,\nu_m)}\norm[\rma,\rma,h]{{\cdot}}^2.
\end{equation}
To prove \eqref{eq:var.glob.est.pressure}, we develop $\rmA_h$ in the error equation \eqref{eq:error.equation}, with $\underline{y}_h=(\ulbmvh,\ulbmwh,0,0)$, to see that
\begin{multline}\label{eq:error.pressures.1}
\nu_k\rma_h(\ul{\bme}_u,\ulbmvh)+\nu_m\rma_h(\ul{\bme}_b,\ulbmwh)+\rmd_h(\ulbmvh,e_q)
+\rmd_h(\ulbmwh,e_r)\\
+\rmT_h(\underline{\omega}_h,\underline{x}_h^\sharp,\underline{y}_h)+\rmT_h(\underline{x}_h^\flat,\underline{\omega}_h,\underline{y}_h)=(\rmF_h^\sharp-\rmF_h^\flat)(\underline{y}_h).
\end{multline}
Using Cauchy--Schwarz inequalities, the estimate \eqref{eq:var.glob.est1} and the norm equivalence \eqref{eq:ah.stab.and.bound}, we have
\begin{equation}\label{eq:error.pressure.2}
(1-\chi)\left|\nu_k\rma_h(\ul{\bme}_u,\ulbmvh)+\nu_m\rma_h(\ul{\bme}_b,\ulbmwh)\right|\\
\le C_1\norm[\bmX,\star]{\rmF_h^\sharp-\rmF_h^\flat}\norm[1,1,h]{(\ulbmvh,\ulbmwh)}
\end{equation}
and, invoking the boundedness \eqref{eq:bound.rmTh} of $\rmT_h$,
\begin{multline}\label{eq:error.pressure.3}
(1-\chi)\left|\rmT_h(\underline{\omega}_h,\underline{x}_h^\sharp,\underline{y}_h)+\rmT_h(\underline{x}_h^\flat,\underline{\omega}_h,\underline{y}_h)\right|\\
\le C_2 \norm[\bmX,\star]{\rmF_h^\sharp-\rmF_h^\flat}\left(\norm[1,1,h]{(\ulbmuh^\sharp,\ulbmbh^\sharp)}+\norm[1,1,h]{(\ulbmuh^\flat,\ulbmbh^\flat)}\right)
\norm[1,1,h]{(\ulbmvh,\ulbmwh)}
\end{multline}
with $C_1,C_2$ having the same dependencies as $C_0$ in the theorem. The estimate \eqref{eq:var.glob.est.pressure} follows plugging \eqref{eq:error.pressure.2}--\eqref{eq:error.pressure.3} into \eqref{eq:error.pressures.1} to get an upper bound on $\rmd_h(\ulbmvh,e_q) +\rmd_h(\ulbmwh,e_r)$, and by using the inf--sup property \eqref{eq:dh.boundedness}.
\end{proof}

\begin{corollary}[Existence and uniqueness of the solution]\label{cor:estimates.uniqueness}
The following properties hold:
\begin{itemize}
\item \emph{Existence and a priori estimates}: there exists a solution to \eqref{eq:var.glob}, and any
solution $\underline{x}_h$ to this problem satisfies
\begin{align}\label{eq:glob.var.bound.vel.aa}
\norm[\rma,\rma,h]{(\ulbmuh,\ulbmbh)}\le{}& \frac{C_\rma^\frac12}{\min(\nu_k,\nu_m)^\frac12}\norm[\bmX,\star]{\rmF_h},\\
\label{eq:glob.var.bound.vel.11}
\norm[1,1,h]{(\ulbmuh,\ulbmbh)}
\le{}&
\frac{C_\rma}{\min(\nu_k,\nu_m)}\norm[\bmX,\star]{\rmF_h},\\
\label{eq:glob.var.bound.pres}
\norm[L^2(\Omega)]{q_h}+\norm[L^2(\Omega)]{r_h}\le{}& C_0\norm[\bmX,\star]{\rmF_h}\left(1+\frac{C_\rma}{\min(\nu_k,\nu_m)}\norm[\bmX,\star]{\rmF_h}\right).
\end{align}
\item \emph{Uniqueness of the solution}: if
\begin{equation}\label{eq:var.glob.small2}
\norm[\bmX,\star]{\rmF_h}\le \frac{\chi\min(\nu_k,\nu_m)^2}{\sqrt{2}C_{\rmt}C_{\rma}^2},
\end{equation}
then \eqref{eq:var.glob} has a unique solution.
\end{itemize}
\end{corollary}

\begin{proof}
The a priori estimates \eqref{eq:glob.var.bound.vel.aa}--\eqref{eq:glob.var.bound.pres} follow from Lemma \ref{lem:stab.variational} with the pair of solution/source given by $(\underline{x}_h^\sharp,\rmF_h^\sharp)=(\ul{0},0)$ and $(\underline{x}_h^\flat,\rmF_h^\flat)=(\underline{x}_h,\rmF_h)$. Indeed, \eqref{eq:var.glob.small1} holds with $\chi=0$ and \eqref{eq:var.glob.est1} together with \eqref{eq:bound.11.aa} yields \eqref{eq:glob.var.bound.vel.aa} and \eqref{eq:glob.var.bound.vel.11}. 
The estimate \eqref{eq:glob.var.bound.pres} directly follows from this bound and \eqref{eq:var.glob.est.pressure}.

\medskip

To prove the existence of a solution, we use Lemma \ref{lem:topological.degree} below. The space $\bmX_h^k$ is a Euclidean space equipped with the norm
	\[
		\norm[\bmX, h]{(\ulbmvh, \ulbmwh, s_h, z_h)}^2 \defeq \norm[1,1,h]{(\ulbmvh,\ulbmwh)}^2 + \norm[\Omega]{s_h}^2 + \norm[\Omega]{z_h}^2,
	\]
	whose inner product is denoted by $(\cdot,\cdot)_{\bmX,h}$.
	For each $\rho\in[0,1]$, we define a function $\Psi_\rho : \bmX_h^k \to \bmX_h^k$ such that for all $\underline{x}_h\in\bmX_h^k$, $\Psi_\rho(\underline{x}_h)$ is the unique element of $\bmX_h^k$ satisfying 
	\[
	(\Psi_\rho(\underline{x}_h),\underline{y}_h)_{\bmX,h}=\rmA_h(\underline{x}_h,\underline{y}_h)+\rho\rmT_h(\underline{x}_h,\underline{x}_h,\underline{y}_h)-\rmF_h(\underline{y}_h)\quad\forall\underline{y}_h\in\bmX_h^k.
	\]
	We shall now show that $\Psi_{\rho}$ satisfies each of the conditions in Lemma \ref{lem:topological.degree}.
	\begin{enumerate}[label=(\roman*)]
		\item As $\bmX_h^k$ is finite dimensional, and each of the forms $\rmA_h$, $\rmT_h$ and $\rmF_h$ are continuous, then $\Psi_{\rho}$ is also continuous.
		\item Suppose $\rho\in[0,1]$ and $(\ulbmuh, \ulbmbh, q_h, r_h)\in\bmX_h^k$ are such that $\Psi_{\rho}(\ulbmuh, \ulbmbh, q_h, r_h) = (\ulbm{0}, \ulbm{0}, 0, 0)$.
		Noting that $\rho\rmT_h$ satisfies the same properties as $\rmT_h$ (with the same constant $\sqrt{2}C_{\rmt}$ in \eqref{eq:bound.rmTh}), we can invoke \eqref{eq:glob.var.bound.vel.11} and \eqref{eq:glob.var.bound.pres} to infer the existence of
$\mu>0$ independent of $\rho$ such that $\norm[\bmX, h]{(\ulbmuh, \ulbmbh, q_h, r_h)} < \mu$.
		\item It is clear that $\Psi_0$ is an affine function and $\Psi_0((\ulbmuh, \ulbmbh, q_h, r_h)) = (\ulbm{0}, \ulbm{0}, 0, 0)$ describes a decoupled Stokes problem for each of $(\ulbmuh, q_h)$ and $(\ulbmbh, r_h)$ for which solutions are known to exist. Upon increasing $\mu$ we also have $\norm[\bmX, h]{(\ulbmuh, \ulbmbh, q_h, r_h)} < \mu$.
	\end{enumerate}
	We may now invoke Lemma \ref{lem:topological.degree} to establish the existence of a solution to $\Psi_1(\underline{x}_h)=0$, which is therefore a solution to \eqref{eq:var.glob}.

\medskip

Finally, to establish the uniqueness of this solution under assumption \eqref{eq:var.glob.small2}, take two solutions $(\underline{x}_h^\sharp,\underline{x}_h^\flat)$ for the same right-hand side $\rmF_h^\sharp=\rmF_h^\flat=\rmF_h$ and assume that \eqref{eq:var.glob.small2} holds. Then \eqref{eq:glob.var.bound.vel.11} shows that \eqref{eq:var.glob.small1} holds.
The bounds \eqref{eq:var.glob.est1} and \eqref{eq:var.glob.est.pressure} then give $\underline{x}_h^\sharp=\underline{x}_h^\flat$.
\end{proof}

\begin{lemma}[{\cite[Lemma 9.6]{di-pietro.droniou:2020:hybrid}}]\label{lem:topological.degree}
	Let $W$ be a finite-dimensional vector space equipped with a norm $\norm[W]{\cdot}$, and let a function $\Psi:W\times[0,1]\to W$ satisfy the following assumptions:
	\begin{enumerate}[label=(\roman*)]
		\item $\Psi$ is continuous;
		\item There exists $\mu > 0$ such that, for any $(w,\rho) \in W\times[0,1]$,
		\[
		\Psi(w, \rho) = 0 \implies \norm[W]{w} \ne \mu;
		\]
		\item $\Psi(\cdot, 0)$ is an affine function and the equation $\Psi(w, 0) = 0$ has a solution $w\in W$ such that $\norm[W]{w} < \mu$.
	\end{enumerate}
	Then there exists $w\in W$ such that $\Psi(w, 1) = 0$ and $\norm[W]{w} < \mu$.
\end{lemma}

\subsection{Proof of the main results}

\begin{proof}[Proof of Theorem \ref{thm:existence}]
Follows directly from Corollary \ref{cor:estimates.uniqueness} and applying \eqref{eq:discrete.poincare} to bound the dual norm of $\rmF_h$.
\end{proof}

\begin{proof}[Proof of Theorem \ref{thm:convergence}]
Owing to \eqref{eq:discrete.apriori.bound}, $(\norm[1,h]{\ulbmuh})_{h\in\mathcal H}$, $(\norm[1,h]{\ulbmbh})_{h\in\mathcal H}$, $(\norm[\Omega]{q_h})_{h\in\mathcal H}$ and $(\norm[\Omega]{r_h})_{h\in\mathcal H}$ are
bounded. As noted in the proof of Theorem \ref{th:discrete.rellich}, the discrete compactness result \cite[Theorem 9.29]{di-pietro.droniou:2020:hybrid} for Dirichlet boundary conditions also holds under Assumption \ref{assum:star.shaped}. Using this result (for the velocity field) and Theorem \ref{th:discrete.rellich} (for the magnetic field), the bounds mentioned above allow us to extract subsequences such that
\begin{equation}\label{eq:weak.convergences}
\begin{aligned}
\bmuh\to \bmu\mbox{ and }\bmbh\to\bmb&\quad\mbox{ in $L^s(\Omega)^3$ for all $s\in[1,6)$},\\
\nabla_h \rh{k+1}\ulbmuh\to \nabla\bmu\mbox{ and }\nabla_h\rh{k+1}\ulbmbh\to\nabla\bmb&\quad\mbox{ weakly in $L^2(\Omega)^{3\times 3}$},\\
\Gh{l}\ulbmuh\to \nabla\bmu\mbox{ and }\Gh{l}\ulbmbh\to\nabla\bmb&\quad\mbox{ weakly in $L^2(\Omega)^{3\times 3}$, for all integers $l\ge 0$},\\
q_h\to q\mbox{ and }r_h\to r&\quad\mbox{ weakly in $L^2(\Omega)$},
\end{aligned}
\end{equation}
where $\bmu\in\bmU$, $\bmb\in\bmB$,
and $(q,r)\in P^2$ (the zero average condition following from the weak convergences of $q_h,r_h$, which have zero average).

Take $\bmv\in\bmC^\infty_c(\Omega)$, $\bmw\in\bmC^\infty_{\bmn}(\Omega):=\{\bmz\in\bmC^\infty(\overline{\Omega})\,:\,\bmz\cdot\bmn=0\mbox{ on $\partial\Omega$}\}$, and $(s,z)\in (C^\infty_c(\Omega)\cap P)^2$,
and use $(\bmIhk\bmv,\bmIhk\bmw,\pihzr{k}s,\pihzr{k}z)$ as test functions in the discrete scheme \eqref{eq:discrete}.
The convergences \eqref{eq:weak.convergences} and the strong convergence of the interpolates in $H^1(\Omega)^3$ (see \cite[Proposition 9.31]{di-pietro.droniou:2020:hybrid}) enable us to reason as in Step 2 of the proof of \cite[Theorem 9.32]{di-pietro.droniou:2020:hybrid} to pass to the limit in the linear and non-linear terms of the scheme, to see that $(\bmu,\bmb,q,r)$ satisfies \eqref{eq:variational} for test functions in $\bmC^\infty_c(\Omega)\times \bmC^\infty_{\bmn}(\Omega)\times (C^\infty_c(\Omega)\cap P)^2$. Since this space is dense in $\bmU\times\bmB\times P^2$ ($\Omega$ being polyhedral, the density of $\bmC^\infty_{\bmn}(\Omega)$ in $\bmB$ can be established following the approach in \cite{D02-2}), this establishes that $(\bmu,\bmb,q,r)$ is a solution to the continuous problem.

To prove the strong convergence of the gradients of the velocity and the magnetic field, we use $(\ulbmuh,\ulbmbh, q_h, r_h)$ as test functions in the scheme \eqref{eq:var.glob}. Owing to \eqref{eq:coer.rmAh} and \eqref{eq:skew.sym.rmTh}, this yields
\begin{equation}\label{eq:energy1}
\nu_k\rma_h(\ulbmuh,\ulbmuh)+\nu_m\rma_h(\ulbmbh,\ulbmbh)=(\bmf,\bmuh)_\Omega+(\bmg,\bmbh)_\Omega.
\end{equation}
By the convergences of $\bmuh$ and $\bmbh$, we have, as $h\to 0$,
\begin{equation}\label{eq:energy2}
(\bmf,\bmuh)_\Omega+(\bmg,\bmbh)_\Omega\to (\bmf,\bmu)_\Omega+(\bmg,\bmb)_\Omega=\nu_k \norm[\Omega]{\nabla\bmu}^2+\nu_m\norm[\Omega]{\nabla\bmb}^2,
\end{equation}
where the equality is obtained plugging $(\bmv,\bmw,s,z)=(\bmu,\bmb,q,r)$ in the continuous weak formulation \eqref{eq:variational}.
On the other hand, by definition of $\rma_h$ it holds $\rma_h(\ulbmvh,\ulbmvh)\ge \norm[\Omega]{\nabla_h\rh{k+1}\ulbmvh}^2$
for all $\ulbmvh\in\bmUhk$.
Combined with \eqref{eq:energy1} and \eqref{eq:energy2}, this yields
\[
\limsup_{h\to 0}\left(\nu_k\norm[\Omega]{\nabla_h\rh{k+1}\ulbmuh}^2+\nu_m\norm[\Omega]{\nabla_h\rh{k+1}\ulbmbh}^2\right)
\le\nu_k \norm[\Omega]{\nabla\bmu}^2+\nu_m\norm[\Omega]{\nabla\bmb}^2.
\]
This inequality shows that the weak convergence of $(\nabla_h\rh{k+1}\ulbmuh,\nabla_h\rh{k+1}\ulbmbh)$ in $(L^2(\Omega)^{3\times 3})^2$ is actually strong.

The strong convergence of the $(q_h,r_h)$ is obtained introducing $(\bmv_q,\bmw_r)\in \bmU\times\bmB$ such that $\DIV\bmv_q=q$, $\DIV\bmw_r=r$, $\norm[H^1(\Omega)^3]{\bmv_q}\lesssim\norm[\Omega]{q}$ and $\norm[H^1(\Omega)^3]{\bmw_r}\lesssim\norm[\Omega]{r}$, and utilising $\bmIhk\bmv_q$ in \eqref{eq:discrete.fluid}, $\bmIhk\bmw_r$ in \eqref{eq:discrete.magnetic}, and reasoning as in Step 4 of the proof of \cite[Theorem 9.32]{di-pietro.droniou:2020:hybrid}.
We omit the details. 
\end{proof}

\begin{proof}[Proof of Theorem \ref{thm:uniqueness}]
Using Corollary \ref{cor:estimates.uniqueness}, we only have to show that \eqref{eq:data.smallness} implies \eqref{eq:var.glob.small2}, which directly follows from
\begin{align*}
|\rmF_h((\ulbmvh,\ulbmwh,0,0))|=\left|(\bmf,\bmvh)_\Omega+(\bmg,\bmwh)_\Omega\right|
\le{}& \left(\norm[\Omega]{\bmf}^2+\norm[\Omega]{\bmg}^2\right)^{\frac12}
\left(\norm[\Omega]{\bmvh}^2+\norm[\Omega]{\bmwh}^2\right)^{\frac12}\\
\le{}& C_p\left(\norm[\Omega]{\bmf}^2+\norm[\Omega]{\bmg}^2\right)^{\frac12}
\norm[1,1,h]{(\bmvh,\bmwh)},
\end{align*}
where the conclusion follows using the Poincar\'e inequality \eqref{eq:discrete.poincare}.
\end{proof}

\begin{proof}[Proof of Theorem \ref{thm:discrete.energy.error}]
We apply Lemma \ref{lem:stab.variational} with $\underline{x}_h^\sharp=(\ulbmuh,\ulbmbh,q_h,r_h)$ solution of the scheme \eqref{eq:discrete} and $\underline{x}_h^\flat=(\bmIhk\bmu,\bmIhk\bmb,\pihzr{k}q,\pihzr{k}r)$ interpolate on $\bmX_h^k$ of the exact solution $(\bmu,\bmb,q,r)$ of the continuous problem \eqref{eq:variational}. 
Decomposing $\bmf$ and $\bmg$ according to \eqref{eq:mhd.fluid} and \eqref{eq:mhd.magnetic}, we see that $\underline{x}_h^\flat$ solves \eqref{eq:var.glob} with source term
\[
\rmF_h^\flat(\underline{y}_h)=\rmF_h^\sharp(\underline{y}_h)-\calE_{k,h}((\bmu,\bmb,q,r);\ulbmvh)-\calE_{m,h}((\bmu,\bmb,q,r);\ulbmwh),
\]
where we recall that the consistency errors $\calE_{k,h}$ and $\calE_{m,h}$ are defined in \eqref{eq:kinetic.cons.def} and \eqref{eq:magnetic.cons.def}.
As established in the proof of Theorem \ref{thm:uniqueness}, the assumption \eqref{eq:data.smallness} shows that $(\ulbmuh^\sharp,\ulbmbh^\sharp)$ satisfies \eqref{eq:var.glob.small1}. The error estimates \eqref{eq:discrete.energy.error.ub} and \eqref{eq:discrete.energy.error.qr} then follow from \eqref{eq:var.glob.est1}--\eqref{eq:var.glob.est.pressure} and the consistency estimates \eqref{eq:kinetic.cons.error}--\eqref{eq:magnetic.cons.error}, together with the a priori estimates \eqref{eq:glob.var.bound.vel.11}, \eqref{eq:continuous.energy} (see below) and the interpolation bound of \cite[Proposition 2.2]{di-pietro.droniou:2020:hybrid} to write $\norm[1,1,h]{(\ulbmuh,\ulbmbh)}+\norm[1,1,h]{(\bmIhk\bmu,\bmIhk\bmb)}\le C\Brac{\norm[\Omega]{\bmf}^2+\norm[\Omega]{\bmg}^2}^{1/2}$ for some $C$ independent of $h$, $\bmf$ and $\bmg$, but dependent on $\nu_k$ and $\nu_m$.
\end{proof}

The following a priori bound on the continuous solution was used in the proof of Theorem \ref{thm:discrete.energy.error} above.

\begin{proposition}[Continuous a priori bound]\label{prop:continuous.apriori.bound}
Any solution to \eqref{eq:variational} satisfies
\begin{equation}
	\nu_k \a(\bmu, \bmu) + \nu_m \a(\bmb, \bmb) \lesssim \nu_k^{-1}\norm[\Omega]{\bmf}^2 + \nu_m^{-1}\norm[\Omega]{\bmg}^2. \label{eq:continuous.energy}
\end{equation}
\end{proposition}

\begin{proof}
The following integration by parts formula holds for all $\bmw,\bmv,\bmz\in\HONE(\Omega)^3$:
\begin{equation}\label{eq:convective.ibp}
	\int_{\Omega}(\bmv \cdot \nabla)\bmw \cdot \bmz + \int_{\Omega}(\bmv \cdot \nabla)\bmz \cdot \bmw + \int_{\Omega}(\nabla \cdot \bmv)(\bmw \cdot \bmz) = \int_{\partial \Omega} (\bmv \cdot \bmn) (\bmw \cdot \bmz).
\end{equation}
Thus, for any $\bmv \in \HONE(\Omega)^3$ such that $\nabla \cdot \bmv = 0$ in $\Omega$ and $\bmv \cdot \bmn = 0$ on $\partial\Omega$,
it holds that
\begin{equation}
	\rmt(\bmv, \bmw, \bmz) = -\rmt(\bmv, \bmz, \bmw) \qquad \forall \bmw, \bmz \in \HONE(\Omega)^3. \label{eq:t.skew.sym}
\end{equation}
Set $\bmv = \bmu$ in \eqref{eq:variational.fluid}, $\bmw = \bmb$ in \eqref{eq:variational.magnetic}, and add the resulting equations to get
\begin{equation}	
\nu_k \a(\bmu, \bmu) + \rmt(\bmu, \bmu, \bmu) - \rmt(\bmb, \bmb, \bmu) 
+\nu_m \a(\bmb, \bmb) + \rmt(\bmu, \bmb, \bmb) - \rmt(\bmb, \bmu, \bmb) =\brac[\Omega]{\bmf, \bmu}+ \brac[\Omega]{\bmg, \bmb},
\label{eq:energy.fluid.magnetic}
\end{equation}
where we have invoked the zero divergence conditions \eqref{eq:variational.fluid.div}--\eqref{eq:variational.magnetic.div} to cancel out the terms $\rmd(\bmu,q)$ and $\rmd(\bmb,r)$. From equation \eqref{eq:t.skew.sym} we infer $\rmt(\bmu, \bmu, \bmu) = \rmt(\bmu, \bmb, \bmb) = 0$ and 
$\rmt(\bmb, \bmb, \bmu) = - \rmt(\bmb, \bmu, \bmb)$. Plugged into \eqref{eq:energy.fluid.magnetic} this yields
\[
	\nu_k \a(\bmu, \bmu) + \nu_m \a(\bmb, \bmb) = \brac[\Omega]{\bmf, \bmu} + \brac[\Omega]{\bmg, \bmb}.
\]
By invoking Cauchy--Schwarz and Poincar\'{e} inequalities (the latter of which holds due to the boundary conditions on $\bmu$, $\bmb$) we readily infer \eqref{eq:continuous.energy}.
\end{proof}

\section{Numerical tests}\label{sec:numerical}

We provide here a variety of numerical tests for the scheme \eqref{eq:discrete} on two families of polyhedral meshes. The method is implemented using the \texttt{HArDCore} open source C++ library available at \url{https://github.com/jdroniou/HArDCore}. The nonlinear algebraic system is resolved via a Newton iterative scheme:
\begin{equation*}
	\bmU^{(n + 1)} = \bmU^{(n)} + \bmdelta^{(n)};
\end{equation*}
\begin{equation}\label{eq:newton.system}
	\rmD\rmG(\bmU^{(n)})\bmdelta^{(n)} = -\rmG(\bmU^{(n)}),
\end{equation}
where $\bmU^{(n)}$ denotes the vector of unknowns at the $n^{\textrm{th}}$ iterate, $\bmdelta^{(n)}$ denotes the Newton step, $\rmG$ is a nonlinear function such that the system \eqref{eq:var.glob} corresponds to $\rmG(\bmU)=\bm{0}$, with Jacobian matrix denoted by $\rmD\rmG$.
The system \eqref{eq:newton.system} is initialised with a vector of zeroes. At each step, we measure the discrete $l^2$ norm of $\rmG(\bmU^{(n)})$ and once its value relative to the initial value is less than $10^{-6}$ we exit the Newton scheme. At each iteration, the degrees--of--freedom (DOFs) of $\ulbmuh$ and $\ulbmbh$ in each element of the system \eqref{eq:newton.system}, and all but one pressure DOF in each element are eliminated via static condensation \cite[cf.][Section 6.2]{di-pietro.ern.ea:2016:discontinuous}. The resulting linear system is solved using the Pardiso solver found in the \texttt{Eigen} library (which internally invokes the Intel Math Kernel Library (Intel MKL)), with documentation available at \url{https://eigen.tuxfamily.org/dox/index.html}. 

The domain is taken to be the unit cube $\Omega= (0,1)^3$. We consider source terms $\bmf, \bmg$ corresponding to the exact solution $(\bmu,\bmb,q,r)$, $\bmu = (u_i)_{i=1}^3$, $\bmb = (b_i)_{i=1}^3$ where
\begin{align*}
	u_1(x_1,x_2,x_3) \eq \sin(\pi x_1)^2 \sin(\pi x_2) \sin(\pi x_3) \sin(\pi (x_2 - x_3)),\nl
	u_2(x_1,x_2,x_3) \eq \sin(\pi x_1) \sin(\pi x_2)^2 \sin(\pi x_3) \sin(\pi (x_3 - x_1)),\nl
	u_3(x_1,x_2,x_3) \eq \sin(\pi x_1) \sin(\pi x_2) \sin(\pi x_3)^2 \sin(\pi (x_1 - x_2)),\nl
	b_1(x_1,x_2,x_3) \eq -\frac12\sin(\pi x_1) \cos(\pi x_2) \cos(\pi x_3),\nl
	b_2(x_1,x_2,x_3) \eq \cos(\pi x_1) \sin(\pi x_2) \cos(\pi x_3),\nl
	b_3(x_1,x_2,x_3) \eq -\frac12\cos(\pi x_1) \cos(\pi x_2) \sin(\pi x_3),\nl
	q(x_1,x_2,x_3) \eq \sin(2\pi x_1) \sin(2\pi x_2) \sin(2\pi x_3), \nl
	r(x_1,x_2,x_3) \eq 0.
\end{align*}
The velocity field $\bmu$ is designed to satisfy $\DIV \bmu = 0$, and $\bmu = \bm{0}$ on $\partial \Omega$, and the magnetic field $\bmb$ (taken from \cite{he.dong.ea:2022:uniform}) is designed to satisfy $\DIV \bmb = 0$, $\bmb \cdot \bmn = 0$ on $\partial \Omega$, and $\bmn \times (\nabla\times \bmb) = \bm{0}$ on $\partial \Omega$.

The stabilisation bilinear form is taken to be
\begin{align*}
	\sT(\ulbmvT, \ulbmwT) ={}& \int_T \nabla(\bmvT - \bmpiTzr{k}\rT{k+1}\ulbmvT) : \nabla(\bmwT - \bmpiTzr{k}\rT{k+1}\ulbmwT) \\
	 &+ \hT^{-1}\int_\bdryT (\bmvFT - \bmpiFTzr{k}\rT{k+1}\ulbmvT) \cdot (\bmwFT - \bmpiFTzr{k}\rT{k+1}\ulbmwT).
\end{align*}

The accuracy of the scheme is measured by the following relative energy errors,
\[
	E_{\a,\bmu,h}^2 \defeq \nu_k \frac{\ah(\ulbmuh - \bmIhk\bmu, \ulbmuh - \bmIhk\bmu)}{\ah(\bmIhk\bmu, \bmIhk\bmu)}\quad\textrm{and}\quad E_{\a,\bmb,h}^2 \defeq \nu_m \frac{\ah(\ulbmbh - \bmIhk\bmb, \ulbmbh - \bmIhk\bmb)}{\ah(\bmIhk\bmb, \bmIhk\bmb)}, 
\]
relative error in the Lagrange multiplier for the fluid equation
\[
	E_{q,h} \defeq \frac{\norm[\Omega]{q_h-\pihzr{k}q}}{\norm[\Omega]{\pihzr{k}q}},
\]
and relative $L^2$ errors, 
\[
E_{0,\bmu,h} \defeq \frac{\norm[0, h]{\ulbmuh - \bmIhk\bmu}}{\norm[0, h]{\bmIhk\bmu}}\quad\textrm{and}\quad E_{0,\bmb,h} \defeq \frac{\norm[0, h]{\ulbmbh - \bmIhk\bmb}}{\norm[0, h]{\bmIhk\bmb}},
\]
where $\norm[0, h]{{\cdot}}:\bmUhkn\to\R$ is an $L^2$-like norm defined for all $\ulbmvh\in\bmUhkn$ via
\[
	\norm[0, h]{\ulbmvh} \defeq \BRAC{\sum_{T\in\Th}\SqBrac{\norm[T]{\bmvT}^2 + \hT \norm[\bdryT]{\bmvFT}^2}}^\frac12.
\]
We also define a fluid pressure error by
\[
	E_{p,h} \defeq \frac{\norm[\Omega]{p_h-p}}{\norm[\Omega]{p}}
\]
where $p_h$ is defined by equation \eqref{eq:discrete.pressure.def} and we consider a mass density $\rho=1$.

For simplicity of presentation we only consider $\nu_k=\nu_m=0.1$. We note that the singular perturbation problem $\nu_k \to 0$, $\nu_m\to0$ comes at a very high computational expense, and requires significant effort in resolving the non-linear algebraic system (quite often and even with relaxation, the iterative algorithms -- whether Newton or pseudo-transient continuation -- do not converge for values $\le 0.1$ of these parameters). Among polytopal methods for MHD, the reference \cite{gleason.peters.ea:2022:divergence} has put some effort into considering these cases, with Reynolds numbers as large as $10^4$; however the numerical tests were only considered in two dimensions, which is computationally much less challenging. Some 3D tests are conducted for the lowest order VEM in \cite{vem:mhd:22}, but only with Reynolds and magnetic Reynolds numbers equal to $1$. Even for linearised MHD, the 3D tests conducted for the DG method in \cite{houston.schotzau.ea:2009:mixed} only consider parameters as low as $\nu_k=\nu_m=0.1$. Numerical tests show that adding a sufficiently large reaction term (with scaling $\sim\min \{\nu_k,\nu_m\}^{-1}$) to the equations -- which is reminiscent of the transient problem with a small time step -- greatly improves the performance of the iterative methods and allows for resolution of the nonlinear system for far smaller parameters. This suggests the transient problem is simpler to solve than the stationary one. However, we only consider the static problem in this work.

\subsection{Tetrahedral meshes}

We consider here a sequence of tetrahedral meshes with data given in Table \ref{table:TetGen}. The mesh parameter $\varrho$ is computed as the maximum ratio of an element diameter to its in-radius (which is also its star-radius as the elements are convex). We test with various values of polynomial degree $k$, and plot the energy errors versus mesh size in Figure \ref{fig:TetGen.energy} and $L^2$ errors in Figure \ref{fig:TetGen.L2}. The energy error seems to converge as predicted in Theorem \ref{thm:discrete.energy.error} (and Section \ref{sec:estimate.fluid.pressure} for the fluid pressure), and the $L^2$ errors on the velocities and magnetic fields appear to enjoy super convergence.

\begin{table}[H]
	\centering
	\pgfplotstableread{data/TetGen_mesh_data.dat}\loadedtable
	\pgfplotstabletypeset
	[
	columns={hMax,NCells,NIFaces,Skew}, 
	columns/hMax/.style={column name=\(h\)},
	columns/NCells/.style={column name=\(\CARD {\Th}\)},
	columns/NIFaces/.style={column name=\(\CARD {\Fh^i}\)},
	columns/Skew/.style={column name=\(\varrho\)},
	every head row/.style={before row=\toprule,after row=\midrule},
	every last row/.style={after row=\bottomrule} 
	]\loadedtable
	\caption{Parameters of tetrahedral meshes.}
	\label{table:TetGen}
\end{table}

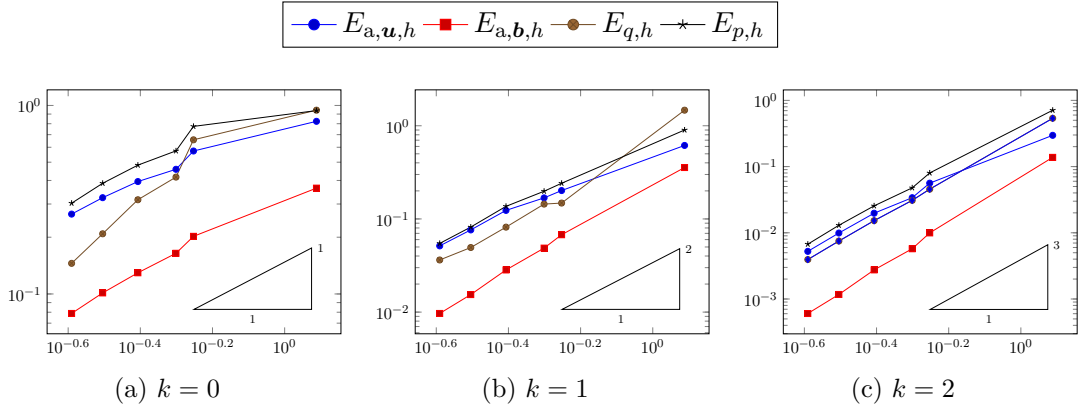
\begin{figure}[H]
\centering
\ref{TetGen:fig1}
\vspace{0.5cm}\\
\subcaptionbox{$k=0$}
{
	\begin{tikzpicture}[scale=0.57]
	\begin{loglogaxis}[ legend columns=4, legend to name=TetGen:fig1 ]
	\addplot table[x=hMax,y=VelocityEnergyError] {data/TetGen_k0_v01.dat};
	\addplot table[x=hMax,y=MagneticEnergyError] {data/TetGen_k0_v01.dat};
	\addplot table[x=hMax,y=PressureEnergyError] {data/TetGen_k0_v01.dat};
	\addplot table[x=hMax,y=FluidPressureError] {data/TetGen_k0_v01_with_fluid_pressure.dat};
	\legend{$E_{\a,\bmu,h}$,$E_{\a,\bmb,h}$,$E_{q,h}$,$E_{p,h}$};
	\logLogSlopeTriangle{0.90}{0.4}{0.1}{1}{black};
	\end{loglogaxis}
	\end{tikzpicture}
}
\subcaptionbox{$k=1$}
{
	\begin{tikzpicture}[scale=0.57]
	\begin{loglogaxis}
	\addplot table[x=hMax,y=VelocityEnergyError] {data/TetGen_k1_v01.dat};
	\addplot table[x=hMax,y=MagneticEnergyError] {data/TetGen_k1_v01.dat};
	\addplot table[x=hMax,y=PressureEnergyError] {data/TetGen_k1_v01.dat};
	\addplot table[x=hMax,y=FluidPressureError] {data/TetGen_k1_v01_with_fluid_pressure.dat};
	\logLogSlopeTriangle{0.90}{0.4}{0.1}{2}{black};
	\end{loglogaxis}
	\end{tikzpicture}
}
\subcaptionbox{$k=2$}
{
	\begin{tikzpicture}[scale=0.57]
	\begin{loglogaxis}
	\addplot table[x=hMax,y=VelocityEnergyError] {data/TetGen_k2_nu01.dat};
	\addplot table[x=hMax,y=MagneticEnergyError] {data/TetGen_k2_nu01.dat};
	\addplot table[x=hMax,y=PressureEnergyError] {data/TetGen_k2_nu01.dat};
	\addplot table[x=hMax,y=FluidPressureError] {data/TetGen_k2_v01_with_fluid_pressure.dat};
	\addplot table[x=hMax,y=PressureEnergyError] {data/TetGen_k2_v01_with_fluid_pressure.dat};
	\logLogSlopeTriangle{0.90}{0.4}{0.1}{3}{black};
	\end{loglogaxis}
	\end{tikzpicture}
}
\caption{Energy Error vs $h$}
\label{fig:TetGen.energy}
\end{figure}

\begin{figure}[H]
	\centering
	\ref{TetGen:fig2}
	\vspace{0.5cm}\\
	\subcaptionbox{$k=0$}
	{
		\begin{tikzpicture}[scale=0.57]
		\begin{loglogaxis}[ legend columns=2, legend to name=TetGen:fig2 ]
		\addplot table[x=hMax,y=VelocityL2Error] {data/TetGen_k0_v01.dat};
		\addplot table[x=hMax,y=MagneticL2Error] {data/TetGen_k0_v01.dat};
		\legend{$E_{0,\bmu,h}$,$E_{0,\bmb,h}$};
		\logLogSlopeTriangle{0.90}{0.4}{0.1}{2}{black};
		\end{loglogaxis}
		\end{tikzpicture}
	}
	\subcaptionbox{$k=1$}
	{
		\begin{tikzpicture}[scale=0.57]
		\begin{loglogaxis}
		\addplot table[x=hMax,y=VelocityL2Error] {data/TetGen_k1_v01.dat};
		\addplot table[x=hMax,y=MagneticL2Error] {data/TetGen_k1_v01.dat};
		\logLogSlopeTriangle{0.90}{0.4}{0.1}{3}{black};
		\end{loglogaxis}
		\end{tikzpicture}
	}
	\subcaptionbox{$k=2$}
	{
		\begin{tikzpicture}[scale=0.57]
		\begin{loglogaxis}
		\addplot table[x=hMax,y=VelocityL2Error] {data/TetGen_k2_nu01.dat};
		\addplot table[x=hMax,y=MagneticL2Error] {data/TetGen_k2_nu01.dat};
		\logLogSlopeTriangle{0.90}{0.4}{0.1}{4}{black};
		\end{loglogaxis}
		\end{tikzpicture}
	}
	\caption{$L^2$ Error vs $h$}
	\label{fig:TetGen.L2}
\end{figure}
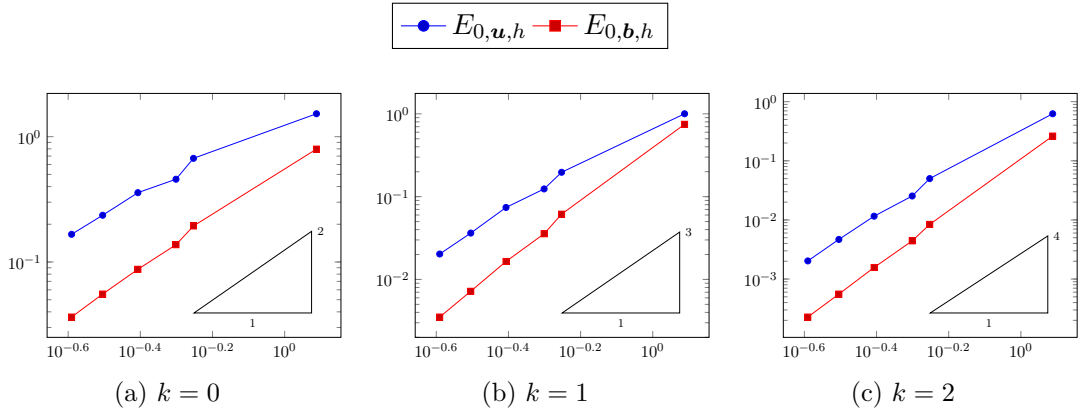

\subsection{Voronoi meshes}

To conclude the tests we run the scheme on a family of Voronoi meshes with data displayed in Table \ref{table:Voro1}. The energy error is plotted in Figure \ref{fig:Voro1.energy} and $L^2$ error is plotted in Figure \ref{fig:Voro1.L2}. It appears that the error in the velocity terms begins to converge sub-optimally before approaching optimal convergence. This effect is particularly apparent for the lowest order case $k=0$, with error plots given in Figure \ref{fig:Voro1.energy} (a) and \ref{fig:Voro1.L2} (a). The reason for this is potentially due to the scheme only approaching, when $k$ is small, the asymptotic convergence rate once the mesh size is sufficiently small.

\begin{table}[H]
\centering
\pgfplotstableread{data/Voro1_mesh_data.dat}\loadedtable
\pgfplotstabletypeset
[
columns={hMax,NCells,NIFaces,Skew}, 
columns/hMax/.style={column name=\(h\)},
columns/NCells/.style={column name=\(\CARD {\Th}\)},
columns/NIFaces/.style={column name=\(\CARD {\Fh^i}\)},
columns/Skew/.style={column name=\(\varrho\)},
every head row/.style={before row=\toprule,after row=\midrule},
every last row/.style={after row=\bottomrule} 
]\loadedtable
\caption{Parameters of Voronoi meshes.}
\label{table:Voro1}
\end{table}

\begin{figure}[H]
	\centering
	\ref{Voro1:fig1}
	\vspace{0.5cm}\\
	\subcaptionbox{$k=0$}
	{
		\begin{tikzpicture}[scale=0.57]
		\begin{loglogaxis}[ legend columns=4, legend to name=Voro1:fig1 ]
		\addplot table[x=hMax,y=VelocityEnergyError] {data/Voro1_k0_v01.dat};
		\addplot table[x=hMax,y=MagneticEnergyError] {data/Voro1_k0_v01.dat};
		\addplot table[x=hMax,y=PressureEnergyError] {data/Voro1_k0_v01.dat};
		\addplot table[x=hMax,y=FluidPressureError] {data/Voro1_k0_v01_with_fluid_pressure.dat};
		\legend{$E_{\a,\bmu,h}$,$E_{\a,\bmb,h}$,$E_{q,h}$,$E_{p,h}$};
		\logLogSlopeTriangle{0.90}{0.4}{0.1}{1}{black};
		\end{loglogaxis}
		\end{tikzpicture}
	}
	\subcaptionbox{$k=1$}
	{
		\begin{tikzpicture}[scale=0.57]
		\begin{loglogaxis}
		\addplot table[x=hMax,y=VelocityEnergyError] {data/Voro1_k1_v01.dat};
		\addplot table[x=hMax,y=MagneticEnergyError] {data/Voro1_k1_v01.dat};
		\addplot table[x=hMax,y=PressureEnergyError] {data/Voro1_k1_v01.dat};
		\addplot table[x=hMax,y=FluidPressureError] {data/Voro1_k1_v01_with_fluid_pressure.dat};
		\logLogSlopeTriangle{0.90}{0.4}{0.1}{2}{black};
		\end{loglogaxis}
		\end{tikzpicture}
	}
	\subcaptionbox{$k=2$}
	{
		\begin{tikzpicture}[scale=0.57]
		\begin{loglogaxis}
		\addplot table[x=hMax,y=VelocityEnergyError] {data/Voro1_k2_nu01.dat};
		\addplot table[x=hMax,y=MagneticEnergyError] {data/Voro1_k2_nu01.dat};
		\addplot table[x=hMax,y=PressureEnergyError] {data/Voro1_k2_nu01.dat};
		\addplot table[x=hMax,y=FluidPressureError] {data/Voro1_k2_v01_with_fluid_pressure.dat};
		\logLogSlopeTriangle{0.90}{0.4}{0.1}{3}{black};
		\end{loglogaxis}
		\end{tikzpicture}
	}
	\caption{Energy Error vs $h$, Voronoi meshes}
	\label{fig:Voro1.energy}
\end{figure}
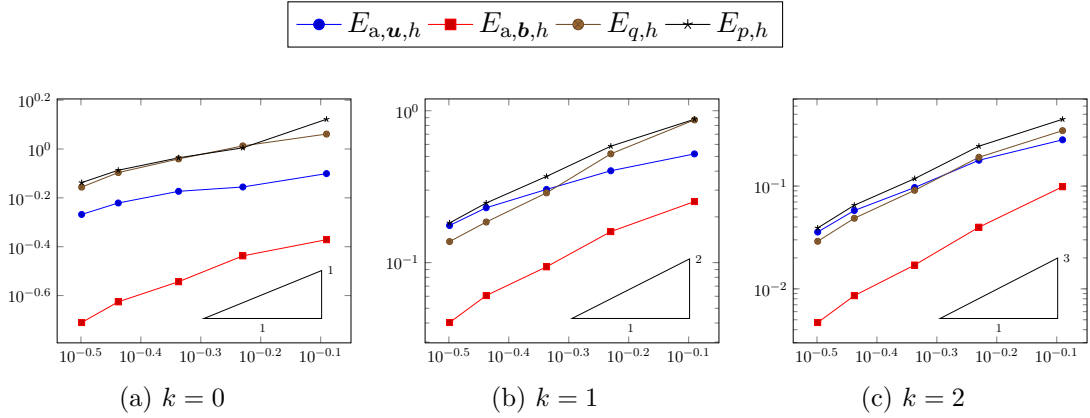

\begin{figure}[H]
	\centering
	\ref{Voro1:fig2}
	\vspace{0.5cm}\\
	\subcaptionbox{$k=0$}
	{
		\begin{tikzpicture}[scale=0.57]
		\begin{loglogaxis}[ legend columns=2, legend to name=Voro1:fig2 ]
		\addplot table[x=hMax,y=VelocityL2Error] {data/Voro1_k0_v01.dat};
		\addplot table[x=hMax,y=MagneticL2Error] {data/Voro1_k0_v01.dat};
		\legend{$E_{0,\bmu,h}$,$E_{0,\bmb,h}$};
		\logLogSlopeTriangle{0.90}{0.4}{0.1}{2}{black};
		\end{loglogaxis}
		\end{tikzpicture}
	}
	\subcaptionbox{$k=1$}
	{
		\begin{tikzpicture}[scale=0.57]
		\begin{loglogaxis}
		\addplot table[x=hMax,y=VelocityL2Error] {data/Voro1_k1_v01.dat};
		\addplot table[x=hMax,y=MagneticL2Error] {data/Voro1_k1_v01.dat};
		\logLogSlopeTriangle{0.90}{0.4}{0.1}{3}{black};
		\end{loglogaxis}
		\end{tikzpicture}
	}
	\subcaptionbox{$k=2$}
	{
		\begin{tikzpicture}[scale=0.57]
		\begin{loglogaxis}
		\addplot table[x=hMax,y=VelocityL2Error] {data/Voro1_k2_nu01.dat};
		\addplot table[x=hMax,y=MagneticL2Error] {data/Voro1_k2_nu01.dat};
		\logLogSlopeTriangle{0.90}{0.4}{0.1}{4}{black};
		\end{loglogaxis}
		\end{tikzpicture}
	}
	\caption{$L^2$ Error vs $h$, Voronoi meshes}
	\label{fig:Voro1.L2}
\end{figure}
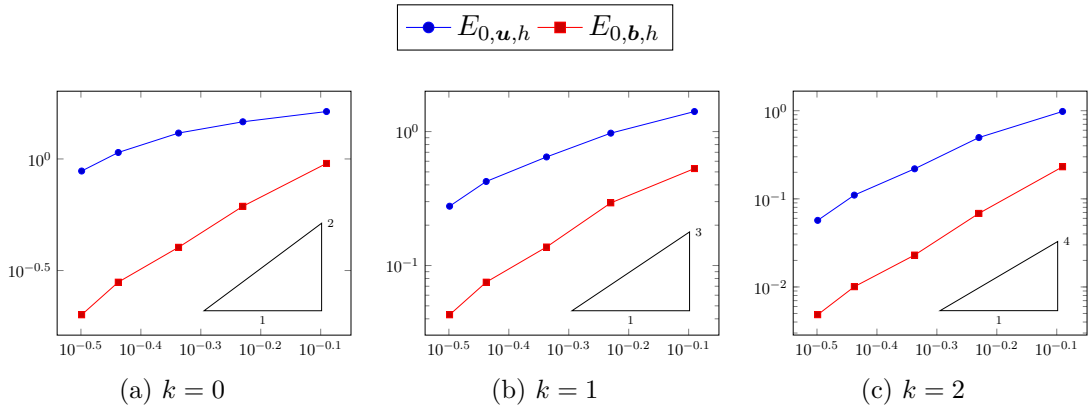

\section{Conclusion}\label{sec:conclusion}

In this paper we have introduced and analysed a novel HHO method for the incompressible MHD equations. We consider a formulation of the problem with the Laplacian as the key differential operator for both the kinetic and magnetic equations. Such an approach allows the design of a scheme where the key spaces are $H^1$, and we show the method to be consistent in any polyhedral domain. The method locally preserves a discrete-divergence free constraint of the kinetic and magnetic unknowns. We prove uniqueness of the discrete solution and optimal approximation rates under a small data assumption. More generally, the existence of a discrete solution and convergence under general data is proven following a compactness approach. The paper is concluded with some numerical tests on polyhedral meshes.


\printbibliography

\end{document}